\documentclass[12pt,a4paper,leqno]{amsart} 

\advance\textwidth by 3.5cm 
\advance\textheight by 2.5cm
\advance\oddsidemargin by -1.6cm
\advance\evensidemargin by -1.6cm  
  
\usepackage{amsmath,amsthm,amssymb}   
\usepackage{enumitem}     

\theoremstyle{plain}  
\begingroup  
\newtheorem{thm}{Theorem}[section]
\newtheorem{lem}[thm]{Lemma}
\newtheorem{prp}[thm]{Proposition} 
   
\endgroup
\theoremstyle{definition}
\newtheorem{dfn}[thm]{Definition}
\newtheorem{xmp}[thm]{Example}

\newcommand\norm[1]{|\!|#1|\!|}
\newcommand\SetOf[2]{\bigl\{#1\,\bigm|\,#2\bigr\}}

\newcommand\av[2]{{\langle {\vec #1},#2\rangle}}
\newcommand\as[2]{{\langle  #1,#2\rangle}}

\newcommand\RP[1]{{\mathbb {RP}}^{#1}}
\newcommand\bs[1]{{\boldsymbol #1}}
\newcommand\ol[1]{{\overline {#1}}}
\renewcommand\vec[1]{\bs #1}

\newcommand\supp{\hbox{\rm supp}}

\newcommand\PG{\hbox{PG}}

\newcommand\co{\operatorname{co}}
\newcommand\sing{\operatorname{ss}}
\newcommand\diag{\operatorname{diag}}

\newcommand\C{{\mathbb C}}

\newcommand\F{{\mathbb F}}

\newcommand\R{{\mathbb R}}
\newcommand\T{{\mathbb T}}

\newcommand\Fn{{\mathbb F}_{(n)}}
\newcommand\Cn{{\mathbb C}_{(n)}}
\newcommand\Hn{{\cH}_{(n)}}

\newcommand\cA{{\mathcal A}}
\newcommand\cB{{\mathcal B}}
\newcommand\cD{{\mathcal D}}

\newcommand\cF{{\mathcal F}}
\newcommand\cH{{\mathcal H}}
\newcommand\cL{{\mathcal L}}

\newcommand\cS{{\mathcal S}}
\newcommand\cT{{\mathcal T}}

\newcommand\cW{{\mathcal W}}

\newcommand\pd{\partial}

\renewcommand\d{{\delta}}
\newcommand\e{{\epsilon}}
\newcommand\g{{\gamma}}
\newcommand\G{{\Gamma}}
\newcommand\om{{\omega}}

\newcommand\s{{\sigma}}

\newcommand\var{{\,\cdot\,}}
\newcommand\z{{\zeta}}
\newcommand\bA{{\bs A}}
\newcommand\bB{{\bs B}}
\newcommand\bx{\bs x}
\newcommand\by{\bs y}
\newcommand\xchi{\smash{\raise 0.5ex\hbox{$\chi$}}}

\author{Brian Jefferies}
\title{Kippenhahn Varieties and the Weyl Calculus for Several Matrices I}

 \address{School of Mathematics\\ NRISAWD\\ NSW 2484
  AUSTRALIA}
 \email{brian.jefferies@gmail.com}

\begin{document}
\baselineskip 16pt
\maketitle

\begin{abstract} 
The paper reviews properties of the Weyl functional calculus for several operators and its relation to the generalised numerical range of n hermitian matrices. The support and singular support of the Weyl functional calculus for $n$ hermitian matrices are determined by Kippenhahn varieties in algebraic geometry.
\end{abstract}

{\centering\footnotesize Dedicated to the memory of Alan McIntosh.\par}
\section{Introduction}
Let $n$ be an  integer and $\vec A =
(A_1,\dots,A_n)$ an $n$-tuple of linear operators densely defined in a complex Banach space $X$ 
whose joint spectra satisfy the spectral reality condition 
\begin{equation}\label{eq:hyp}
\s(\av A\xi)\subset\R,\qquad \xi\in\R^n,
\end{equation}
with $\av A\xi = A_1\xi_1+\dots+A_n\xi_n$. Such a system $\vec A$ is called {\it hyperbolic}.

The term comes from the case when $\bA$ consists of $N\times N$ matrices with $N= 2,3,\dots$. With the characteristic
polynomial of a matrix $B$ defined by $p_B(z) =\det(B-zI)$ for  $z\in\C$, the $n$-tuple $\bA$ is hyperbolic
exactly when the only solutions $z$ of the equation $p_{\av A\xi}(z)=0$ are real for any $\xi\in\R^n$, $\xi\ne 0$,
or that $P^\bA:\xi\longmapsto \det(\xi_0I + \av A{\bs \xi})$, $\xi\in \R^{n+1}$, is a 
homogeneous {\it hyperbolic polynomial} with respect to $e_0 =(1,0,\dots,0)\in \R^{n+1}$. Here the vector $\xi =(\xi_0,\dots,\xi_n)\in \R^{n+1}$ has been written 
as $\xi = \xi_0e_0 +\vec\xi$ with $\vec\xi=\sum_{j=1}^n\xi_je_j$ for the standard basis $(e_0,\dots,e_n)$ of $\R^{n+1}$
and the same notation is used for $\z\in \C^{n+1}$.

General hyperbolic polynomials $p:\R^{n+1}\to \R$ were studied by L. G\aa rding \cite{Gar} in relation to hyperbolic partial differential equations in order that the equation $p(\tau,\vec D) = \d_0$ be well-posed in the sense of distributions with respect to
the differential operators 
$$\tau =  \frac1i\frac{\pd}{\pd t},\qquad \vec D=\bigg(\frac1i\frac{\pd}{\pd x_1},\dots,\frac1i\frac{\pd}{\pd x_n}\bigg).$$
The unique distributional solution $F_p$ of $p(\tau,\vec D) = \d_0$ \cite[Theorem 12.5.1]{H2} may then be expressed by a formula of Herglotz-Petrovsky-Leray 
systematically studied in \cite{ABG1},\cite{ABG2} and \cite[Chapter 12]{H2}. The associated hyperbolic system
\begin{equation}\label{eq:Weyl}
I\frac{\pd }{\pd t} + \av A\nabla  = i\delta_0I
\end{equation}
then has a unique solution $F_\bA = \Phi(\bA)F_{P^\bA}$ for a matrix differential operator $\Phi(\bA)$ of order $N-1$
in $(\tau,\vec D)$
formed by multiplying the Fourier transform of $F_{P^\bA}$ by the $(N-1)$ matrix minors  of $\xi_0I+\av A{\vec\xi}$, $\xi=\xi_0e_0+\vec\xi\in\R^{n+1}$.

The Fourier transform of the Schwartz function $f\in \cS(\R^n)$ is taken to be
$$\hat f(\xi) = \int_{\R^n}e^{-i\as x\xi}f(x)\,dx,\quad \xi\in\R^n.$$
On the other hand, taking the Fourier transform of equation (\ref{eq:Weyl}) in the space variables only
for $t>0$
and solving for the initial value problem, we get the
distribution $\cW_{t\bA}$, $t\in\R$. At $t=1$, the operator valued distribution $\cW_{\bA}$ given by
\begin{equation}\label{eq:Wfc}
\as{\cW_{\bA}} f\ = \frac1{(2\pi)^n}\int_{\R^n} e^{i\av A\xi}\hat f(\xi)\,d\xi,\quad f\in \cS(\R^n),
\end{equation}
is the {\it Weyl functional calculus} for $\bA$ studied since the late 60's by E. Nelson \cite{Nel}, M. Taylor \cite{T1}
and R.F.V. Anderson \cite{A1,A2}. This expression has the advantage that it does not involve determinants
so $\cW_{\bA}$ makes sense
for any system $\bA$ of collectively densely defined linear operators which satisfy an exponential estimate of the
form 
\begin{equation}\label{eq:type}
\|e^{i\av A\xi}\| \le C(1+|\xi|)^s,\quad \xi\in\R^n,
\end{equation}
for some $C,s > 0$, for example, selfadjoint operators defined in a Hilbert space with $s=0$. 
A single bounded linear operator satisfying the estimate $(\ref{eq:type})$ is a generalised scalar operator
\cite[Theorem 5.4.5]{CF}. More precisely, if $\bA=(A_1,\dots,A_n)$ are densely defined selfadjoint operators such that every real linear
combination has a selfadjoint closure $\av A\xi$ on the intersection of the relevant domains,
then (\ref{eq:type}) follows with $C=1$ and $s=0$ from successive applications of the Lie-Kato-Trotter product formula
(see \cite[Chapter III, Corollary 5.8]{EN})

The prime candidate for the Weyl functional calculus
is the system $(\bs X,\bs D)$ of densely defined operators. Here $\bs X= (X_1,\dots, X_n)$ with $X_j$ the operator of multiplication by the variable $x_j$,
$j=1,\dots,n$, so that $\cW_{(\bs X,\bs D)}$ is a $\cL(L^2(\R^n))$-valued Schwartz distribution on $\R^{2n}$.
This is the original Weyl functional calculus and much studied as a pseudodifferental operator \cite{HW},\cite[VII \S14]{Taylor}. The more general definition of $\cW_{\bA}$ for selfadjoint
operators $\bA$ formulated by E. Nelson \cite{Nel} was inspired by ideas of R. Feynman on a general
operator calculus and ``disentangling" procedure \cite{F2},\cite{JLN}. In the case that $\bA$ are selfadjoint elements of a von Neumann
algebra with a given trace $\rho$, then the scalar distribution $\rho\circ \cW_\bA$ is the {\it Wigner transform} \cite{Case, SW}.

In a similar fashion, the Kohn-Nirenberg functional calculus $\s\mapsto\s(\vec X,\vec D)_{\text{KN}}$ defined for 
 symbols $\s\in \cS(\R^{2n})$ by
$$\s(\vec X,\vec D)_{\text{KN}} = \frac1{(2\pi)^{2n}}\int_{\R^{2n}}e^{i\av X q}e^{i\av D p}\hat\s(q,p)\,dq\,dp$$
is an important tool in the analysis of partial differential equations \cite[VII (1.10)]{Taylor}.

If  $\bA$ are bounded linear operators on a Banach space $X$ satisfying the exponential bound (\ref{eq:type}) then
by the Paley-Wiener theorem $\cW_{\bA}$ has compact support and there exist $r > 0$ such that
\begin{equation}\label{eq:type1}
\|e^{i\av A\z}\| \le C'(1+|\Re\z|)^{s'}e^{r|\Im \z|},\quad \z\in\C^n,
\end{equation}
for some $C',s'>0$. The operator valued distribution $\cW_\bA$ has support in the ball $B_r$ of radius $r$ centred at zero.
Any number $r$ strictly greater than $\|\bA\|=(\|A_1\|^2+\dots+\|A_n\|^2)^\frac12$ will do.
All such linear operators satisfy the spectral reality condition (\ref{eq:hyp}).

Back to hyperbolic $(N\times N)$ matrices $\bA$, the solution of (\ref{eq:Weyl}) has an expression
by the Herglotz-Petrovsky-Leray formula so the bound (\ref{eq:type1}) is satisfied with
$s'$ at most $N-1$. In the infinite dimensional situation this may no longer be the case.
The paper of Nelson \cite{Nel} also contains a formula  for hermitian matrices 
$\vec A$ acting on $\C^N$ expressing $\cW_\bA$ as a differential operator of order $(N-1)$
acting on the measure $\mu\circ n_{\bA}^{-1}$. Here $\mu$ is the unitarily invariant probability
measure on the unit sphere $S(\C^N)$ in $\C^N$ centred at zero and
\begin{equation}\label{eq:jnr}
n_\bA:h\longmapsto (\langle A_1h,h\rangle,\dots,\langle A_nh,h\rangle ),\qquad h\in S(\C^N),
\end{equation}
 is the {\it joint numerical range map}. The Payley-Wiener theorem shows that $\cW_\bA$ is supported
 by the convex hull of the {\it joint numerical range} $N_\bA= n_\bA(S(\C^N))$, even for unbounded selfadjoint operators (with a common dense domain
so that $\av A\xi$ has a selfadjoint closure for all $\xi\in\R^n$).

The joint numerical range $N_\bA$ of an $n$-tuple $\bA$ of hermitian matrices has attracted recent interest
because it arises in many areas such as optimisation theory, quantum optics, quantum statistical
mechanics, Wigner transforms and quantum error correction, see \cite{PSW} for an overview. In the
Herglotz-Petrovsky-Leray formula the joint numerical range $N_\bA$ manifests as the trace of
the propagation cone $K(P^\vec A)$ at time $t=1$ and the fundamental formula
\begin{equation}\label{eqn:nrm}
\mu\circ  n_{\bA}^{-1} = (-i)^N(N-1)!F_{P^\bA}(1,\var)
\end{equation}
 holds as distributions linking the Nelson and Herglotz-Petrovsky-Leray representations for systems of 
 hermitian matrices. It is not immediately obvious why $F_{p}(1,\var)$ should be a distribution of order zero
 in the case that $p = P^\bA$ is a hyperbolic determinantal polynomial for a system $\bA$
 of hermitian matrices, see Theorem \ref{thm:5.2} below.
 
 For the case $n=2$ of two hermitian matrices $\bA =(A_1,A_2)$, the set $N_\bA$ may be identified with the usual
 numerical range in $\C$ of the single matrix $A=A_1+iA_2$ under the identification $j:(x,y)\mapsto x+iy$, $x,y\in\R$.
 Then $jN_\bA$ is the convex hull of an algebraic curve $C(A)=jC(\bA)$ studied by R. Kippenhahn \cite{Kip}.
 In this case M. Atiyah, R. Bott and L. G\aa rding show that the singular support $\sing(\cW_\bA)$ of the Weyl
 calculus for two hermitian matrices $\bA$ is exactly $C(\bA)$ \cite[Theorem 14.20]{ABG2}.
 That $\sing(\cW_\bA) \subset C(\bA)$ was shown by J. Bazer and D. Yen \cite{BY2}.
 None of these authors reference the numerical range studies of Kippenhahn \cite{Kip}. Bazer and Yen
 manage to avoid discussing hyperbolic polynomials by invoking a plane wave decomposition of the
 distribution $\cW_{t\bA}$ whose density with respect to Lebesgue measure is the ``Riemann matrix" 
for equation (\ref{eq:Weyl}).

At this stage it is prudent to introduce some modern developments. The {\it Cauchy transform} of a Schwartz distribution
$T\in \cS'(\R)$ is the function $\tilde T:z\mapsto\as T{g_z}$ with 
$g_z(x)=\frac1{2\pi i}\frac1{x-z}$, $z\in\C\setminus\R$, $x\in\R$ \cite[\S 5.1]{Brem}. Then
\begin{equation}\label{eqn:Brem}
T  =\lim_{\e\to 0+} \tilde T(x+i\e)-\tilde T(x+i\e)
\end{equation} 
in the sense of distribution  \cite[\S 5.6]{Brem}.
If we want a similar formula for distributions on $\R^n$ like the Weyl functional calculus $\cW_\bA$, then we need an analogue of the 
normalised Cauchy kernel $\frac{1}{2\pi i}\frac1 z$, $z\in \C$, $z\ne 0$, in higher dimensions.
Then we may obtain a representation for for the density of $\cW_\bA$ with respect to Lebesgue measure
like the Riemann matrix of Bazer and Yen \cite{BY1,BY2}. This is the point where Clifford analysis enters
as a functional calculus technique.

The {\it Clifford algebra} $\C_{(n)}$ over the field $\C$ is generated by standard basis vectors \\
$(e_0,e_1,\dots,e_n)$
in $\R^{n+1}$ with multiplication so that $e_0$ is the unit and  the formula $x^2 = -|x|^2$ holds for $x =\sum_{j=1}^nx_je_j$,
$x_j\in \R$. 
Then we obtain $e_je_k = - e_ke_j$ for $j\ne k$, $j,k=1,\dots,n$ and $e_j^2=-1$ for $j=1,\dots,n$.
The {\it Kelvin inverse} of $x\in\R^{n+1}$ is the vector $x^{-1} =\ol x/|x|^2$, $x\ne 0$.

The Cauchy-Riemann operator is $D=\sum_{j=0}^ne_j\frac{\pd}{\pd x_j}$ and for $\Sigma_n=\frac{2\pi^{\frac{n+1}{2}}}{\Gamma\left(\frac{n+1}{2}\right)}$ the function
$$E(x) = \frac1{\Sigma_n}\frac{\ol x}{|x|^{n+1}},\quad x\in \R^{n+1},\ x\ne 0, $$
is the corresponding Cauchy kernel with $DE=0$ on $\R^{n+1}\setminus\{0\}$. The involution is given by $\ol x =x_0-\bx$
for $x=x_0e_0+\bx$ with $x_0\in\R$ and $\bx\in \R^n\equiv \{0\}\times\R^n\hookrightarrow\R^{n+1}$. For product vectors
$\ol {e_je_k} = \ol{e}_k\,\ol{e}_j$, $j\ne k$, $j,k\ne 0$, and so on. It is easily checked that $E= \ol D\G_{n+1}$ for the fundamental
solution $\G_{n+1}$ of the Laplacian operator $\Delta e_0 =D\ol D$ in $\R^{n+1}$ so that $\G_{n+1}= \frac1{\Sigma_n}\frac{1}{|x|^{n-1}}$
and $\Delta \G_{n+1}=\d_0$ for $n\ge 2$. Setting $G_\om(x) = E(\om-x)$, $x\in\R^n$, for any Schwartz distribution $T \in \cS'(\R^n)$ we have
\begin{equation}\label{eq:jump}
T =\lim_{\e\to0+} \tilde T(x+\e e_0) - \tilde T(x-\e e_0)
\end{equation}
in the sense of distributions for the Cauchy transform $\tilde T(\om) = \as T{G_\om} $, $\om\in\R^{n+1}$, $\om_0\neq 0$
\cite[Theorem 27.7]{BDS} . When the Weyl functional calculus $\cW_\bA$ exists for operators $\bA$, 
the operator valued function $\om\mapsto G_\om(\vec A)= \widetilde{\cW_\bA}(\om)$ defined for all 
$\om\in\R^{n+1}$ with $\om_0\ne 0$ is called the {\it Cauchy kernel} for $\bA$.
For any bounded linear operators $\bA$ on a Banach space $X$, the Cauchy kernel $G_\om(\vec A)$ can be defined by a
series expansion for $\om\in \R^{n+1}$ with $|\om|$ large enough but this is not very useful. For $n=1$,
$G_\om(A) = \frac1{2\pi}(j\om -A)^{-1}$ in $\cL(X)$ if $j\om \notin\s(A)$.

Finally, if $\vec A$ is a hyperbolic system of bounded linear operators (or unbounded with a common dense domain
and uniform resolvent bounds) then $G_\om(\vec A)$ may be defined by a {\it plane wave decomposition} which agrees
with the definition in case $\bA$ satisfies the exponential bounds (\ref{eq:type1}) and so $\cW_\bA$ exists.
This central idea is due to Alan McIntosh in 1988 out of which the present investigation grew.

When $n$ is even, the plane wave formula and equation (\ref{eq:jump}) tell us that the distribution $\cW_\bA$ is a constant times the limit
\begin{equation}\label{eq:pwv}
\lim_{\e\to0+}\int\limits_{S^{n-1}}^{}\left(\langle \bs xI -
\vec A,s\rangle -\e s I\right)^{-n} +\left(\langle \bs xI -
\vec A,s\rangle +\e s I\right)^{-n}\/ds 
\end{equation}
in the sense of distributions for $\bx \in\R^n$. Here $S^{n-1}$ is the unit sphere in $\R^n$.
Even for matrices this is a useful observation because (\ref{eq:pwv}) is similar to the plane wave formula
employed by Bazer and Yen \cite{BY1} in their study of the Riemann matrix but with the pleasing advantage that
it also works for $n\ge 2$. The formula for odd $n$ is
\begin{equation}\label{eq:pwvo}
\lim_{\e\to0+}\int\limits_{S^{n-1}}^{}s\left(\langle \bs xI -
\vec A,s\rangle +\e s I\right)^{-n} -s\left(\langle \bs xI -
\vec A,s\rangle -\e s I\right)^{-n}\/ds.
\end{equation}
By perturbing the integrals over $S^{n-1}$ into the complex domain and observing that we have a closed
matrix valued differential form by homogeneity, the limits can be evaluated {\it pointwise}
and converge to the density of the distribution $\cW_\bA$ with respect to Lebesgue measure. This
is essentially the argument of Bazer and Yen in the case $n=2$ for two hermitian matrices.
The method just described is limited to finite dimensional operators but works for all hyperbolic
$n$-tuples $\bA$ of matrices and uses the full force of the Herglotz-Petrovsky-Leray representation
devised by Atiyah, Bott and G\aa rding.

In addition, we need the analogue $C(\vec A)$ of the Kippenhahn curves in dimensions $n =3,4,\dots$
for which $N_\bA =\co(C(\vec A))$. These semi-algebraic sets have only recently been determined by Plaumann, Sinn and Weis \cite{PSW} and promise to have applications to the many areas where joint numerical range is a central concept.
After settling the many issues involved, the position of the singular support $\sing(\cW_\bA)$ of the Weyl
functional calculus with respect to
the {\it boundary generating set} $C(\vec A)$ or {\it Kippenhahn variety} in $\R^n$ is examined. 
\section{The Kippenhahn Curves}
Let $\vec A = (A_1,A_2)$ be a pair of $(N\times N)$ hermitian
matrices. Set $A=A_1+iA_2$. 
As mentioned in the Introduction, an application of the Paley-Wiener Theorem\index{Paley-Wiener Theorem}
yields that the convex  hull of the
support~$\supp(\cW_{\bA})$ of the associated Weyl
distribution~$\cW_{\bA}$ coincides with the
{\it numerical range\index{numerical range}}
\begin{align*}
K(A)&:=jN_\bA\\
&=\SetOf{\langle Ax,x\rangle}{x\in\C^N,\ |x|=1}
\end{align*}
of the matrix $A$. For more precise information
on the location of~$\supp(\cW_{\bA})$ within the
numerical range of~$A$, we need to have a closer look at
the fine structure of~$K(A)$.

Of particular interest are certain plane algebraic
curves associated with~$A$ that were first investigated
by R.~Kippenhahn \cite{Kip} in~1952. We briefly recall the concepts
involved.

Let $\F=\R$ or $\C$. For $0\le k\le 3$, the
{\it Grassmannian\index{Grassmannian}}\/ $G_{3,k}\F$\glossary{$G_{3,k}\F$}, defined as the set of
all $k$-dimensional $\F$-subspaces of $\F^3$, is a compact
analytic \mbox{$\F$-manifold} of dimension $k(3-k)$. It
has a natural topology, induced by the differential
structure of the manifold, which is determined, for
example, by the metric~$h$ on $G_{3,k}\F$ with
$$
h(U,V)\
=\ \sup_{v\in V,\ |v|=1}
    \inf_{\phantom{\mbox{l}}u\in U,\ |u|=1\phantom{\mbox{l}}}
     \norm{u-v}
\qquad\mbox{ for all }U,V\in G_{3,k}\F \, .
$$

The {\it projective plane\index{projective plane}\/}~$\PG(\F^3)$\glossary{$\PG(\F^3)$} over~$\F$ is given by
$$
\PG(\F^3)\ =\ \bigcup_{0\leq k\leq 3}G_{3,k}\F \, .
$$
The one and two dimensional subspaces of~$\F^3$ are
usually called the {\it points}\/ and {\it lines}\/
in~$\PG(\F^3)$,  respectively.

By common abuse of notation we introduce {\it homogeneous
  coordinates\index{homogeneous coordinates}}\/ for the points in~$\PG(\F^3)$ as
$(u_1:u_2:u_3)=\F(u_1,u_2,u_3)$. The coordinates of a vector
in~$\F^3$ are expressed with respect to the standard basis
for~$\F^3$.

A {\it polarity\index{polarity}}\/ of~$\PG(\F^3)$ is a bijection
on~$\PG(\F^3)$ which reverses the inclusion of subspaces
and the square of which equals the identity mapping. The
{\it standard polarity\index{polarity!standard}}\/ $\pi$ is characterised by
$$
u^\pi\ =\ \SetOf{v\in\F^3}{\sum_{j=1}^{3} u_j v_j = 0}
\qquad \mbox{ for all }u\in G_{3,1}\F\,,
$$\glossary{$u^\pi$}
which gives $u^\pi\in G_{3,2}\F$. Using the polarity~$\pi$,
we can also introduce homogeneous coordinates for the lines
in~$\PG(\F^3)$ by setting $[v_1:v_2:v_3]=(v_1:v_2:v_3)^\pi$.

A nonempty subset $C$ of $G_{3,1}\F$ is called a
{\it plane $\F$-algebraic curve\index{algebraic curve}}\/ if it is the zero
locus of a homogeneous $3$-variate polynomial over~$\F$.
The defining polynomial of~$C$  is not uniquely determined:
if~$f$ defines the curve, then so  does, for example, $f^k$
for any $k\geq 1$. However, every  curve~$C$ has a defining
polynomial of minimal degree which is  unique up to a
constant factor. A curve is said to be {\it irreducible}\/
\index{algebraic curve!irreducible}if it has an irreducible defining polynomial. Since a
polynomial ring over a field is a unique factorisation
domain, each algebraic curve~$C$ is the union of finitely
many irreducible curves. If $C_1,\ldots,C_k$ are the
irreducible components of~$C$ with irreducible defining
polynomials $f_1,\ldots,f_k$, then $f=f_1\cdots f_k$ is
a defining polynomial of~$C$ of minimal degree. We call
$f$ a {\it minimal polynomial\index{minimal polynomial} of~$C$}.\/ Note that an
irreducible real algebraic curve is not necessarily
connected.

Let $f$ be a minimal polynomial of the algebraic curve
$$C=\SetOf{u\in G_{3,1}\F}{f(u)=0}.$$ 
A point $u\in C$ is
called {\it singular\index{algebraic curve!singular point}}\/ or a 
{\it singularity}\/ of~$C$ if
$({\partial f}/{\partial u_j})(u)=0$ for $j=1,2,3$. Observe
that $C$ has at most finitely many singular points. These
are the singular points of the irreducible components of~$C$
together with the points of intersection of any two of these
components. A nonsingular point $u\in C$ is called a
{\it simple\index{algebraic curve!simple point}}\/ point of~$C$. The curve~$C$ is the topological
closure of its simple points. Also, to every simple point
$u\in C$, there exists a neighbourhood of~$u$ in which $C$
admits a smooth parametrization.

Let $C$ be an irreducible plane algebraic curve with minimal
polynomial~$f$. At each simple point $u\in C$, we have
a unique tangent line to~$C$ which is given by
$$
\cT_uC\ =\
   \left[\frac{\partial f}{\partial u_1}(u):
           \frac{\partial f}{\partial u_2}(u):
             \frac{\partial f}{\partial u_3}(u)
   \right].
$$\glossary{$\cT_uC$}
If $C$ is not a projective line or a point, then it is
well-known that the set
$$\SetOf{(\cT_uC)^\pi}{u\in C \mbox{~simple}}$$ is
contained in a unique irreducible algebraic curve~$C^*$\glossary{$C^*$},
the so-called {\it dual curve}\index{algebraic curve!dual curve}\/ of~$C$. In fact, since
an algebraic curve has at most finitely many singularities,
the dual curve is the topological closure of the set
$\SetOf{(\cT_uC)^\pi}{u\in C \mbox{ simple}}$. We have
$C^{**}=C$. If $C$ is a projective line, then
$\SetOf{(\cT_uC)^\pi}{u\in C}$  consists of a single
point~$u$ in~$\PG(\F^3)$. In this case, we set
$C^*=\{u\}$ and define $C^{**}$ to be the image under~$\pi$
of the set of all lines in $\PG(\F^3)$ which pass
through~$u$. This again yields $C^{**}=C$. The dual curve
of a general plane algebraic curve $C$ is the union of
the dual curves of its irreducible components. In
particular, $C$ and $C^*$ have the same number of
irreducible components.

In general, it is difficult to derive an explicit equation
for the dual curve $C^*$ from the given equation of a
curve~$C$. However, from the above we obtain the following
criterion for a point in~$\PG(\F^3)$ to belong to~$C^*$.

\begin{lem}\label{lem:4.3.1}
  Let $(x_1:x_2:1)\in G_{3,1}\F$. If there exists a smooth
  local parametrization
  $\zeta\longmapsto (c(\zeta):s(\zeta):\mu(\zeta))$
  of~$C$, for $\zeta$ in an open set $U\subseteq\F$, and
  a point $z\in U$ such that $x_1c(z)+x_2s(z)+\mu(z)=0$
  and $x_1c'(z)+x_2s'(z)+\mu'(z)=0$, then the
  point $(x_1:x_2:1)$ belongs to $C^*$.
\end{lem}

\begin{proof}
  The two points $(c(z):s(z):\mu(z))$ and
  $(c'(z):s'(z):\mu'(z))$ span the tangent line
  $\cT_{(c(z):s(z):\mu(z))}C$ to~$C$
  at $(c(z):s(z):\mu(z))$. The equations
  $x_1c(z)+x_2s(z)+\mu(z)=0$ and
  $x_1c'(z)+x_2s'(z)+\mu'(z)=0$ imply that
  $(x_1:x_2:1)=(\cT_{(c(z):s(z):\mu(z))}C)^\pi$.
  Hence $(x_1:x_2:1)$ belongs to the dual curve~$C^*$
  of~$C$.
\end{proof}

The details and further information on complex algebraic
curves can be found, for example, in~\cite{Sha}. The
literature for the real case is somewhat less easy to
access. As a general reference to the theory of real
algebraic geometry, see~\cite{BCR}.
\medskip

Let $A=A_1+iA_2\in\cL(\C^N)$. Following R.~Kippenhahn~\cite{Kip},
we define the complex algebraic curve~$C_\C(A)$\index{algebraic curve!Kippenhahn|(} in the complex
projective plane $\PG(\C^3)$ by setting its dual curve to be\glossary{$D(A)$}
$$
D(A)\
=\ \SetOf{(c:d:\mu)\in G_{3,1}\C}
                 {\det(cA_1+dA_2+\mu I)=0}.
$$
In~\cite{Kip}, Kippenhahn showed that the real part~$C_\R(A)$
of the curve $C_\C(A)=D(A)^{*}$\glossary{$C_\C(A)$} is contained in the affine
subplane
$F=\SetOf{(\alpha_1:\alpha_2:1)}{(\alpha_1,\alpha_2)\in\R^2}$
of~$\PG(\R^3)$ and, identifying $F$ with $\R^2$ in the
canonical way, that the convex hull~${\rm co}(C_\R(A))$
of~$C_\R(A)$ is precisely the numerical range of~$\vec A$.

The curve $C_\R(A)$\glossary{$C_\R(A)$} considered as a real algebraic curve
in~$\PG(\R^3)$ is the dual curve of the real part of~$D(A)$
given by
$$
D_\R(A)\
=\ \SetOf{(c:d:\mu)\in G_{3,1}\R}
                 {\det(cA_1+dA_2+\mu I)=0}.
$$\glossary{$D_\R(A)$}
Every point $u\in D_\R(A)$ has a representation
$(\cos\theta:\sin\theta:\mu)$ for some
$\theta\in [0,\pi)$ and $\mu\in\R$. As $u$ is a zero of
$\det(cA_1+dA_2+\mu I)$, it  follows that $-\mu$ is
an eigenvalue of the operator
$\cA(\theta)=\cos\theta\,A_1+\sin\theta\,A_2$.

Note that the points in $D_\R(A)$ are in one-to-one
correspondence with the lines $L_{y,t}$ in~$\R^2$ satisfying
$\langle x,t\rangle\in\sigma(\langle \bA,t\rangle)$ for all
$x\in L_{y,t}$. For
$u=(\cos\theta:\sin\theta:\mu)\in D_\R(A)$,
take $t = (\cos\theta,\sin\theta)\in \T$ and $\by\in\R^2$
such that $\langle y,t\rangle=-\mu$. Then $u^\pi$ is
the two dimensional subspace
$$
\bigcup\SetOf{(x_1:x_2:1)}{(x_1,x_2)\in L_{y,t}}
$$
of $\R^3$, that is, $L_{y,t}\times\{1\}$ is the line in
which the plane $u^\pi$ normal to~$u$ in~$\R^3$ cuts the
plane~$\{x_3 = 1\}$.
\section{Clifford Analysis}
The basic idea of forming a 
Clifford algebra\index{Clifford!algebra|(} $\cA$ with $n$ generators
is to take the smallest real or complex algebra $\cA$ with an identity
element $e_0$ such that $\R\oplus\R^n$ is embedded in $\cA$ via the
identification of $(x_0,\bs x)\in\R\oplus\R^n $ with $x_0 e_0+\bs x\in \cA$ 
and the identity
$$\bs x^2 = -|\bs x|^2e_0 = -(x_1^2+x_2^2+\cdots+x_n^2)e_0$$
holds for all $\bs x\in\R^n$. Then we arrive at the following
definition.

Let $\F$ be either the field $\R$ of real numbers or the
field $\C$ of complex numbers. The {\it Clifford algebra\index{Clifford!algebra}} 
$\Fn$\glossary{$\Fn$} over $\F$ is a 
$2^n$-dimensional algebra with unit defined as follows. Given the standard basis vectors
$e_0, e_1,\dots,e_n$ of the vector space $\F^{n+1}$, the basis vectors $e_S$ of
$\Fn$ are indexed by all finite subsets $S$ of $\{1,2,\dots,n\}$. The basis vectors are
determined by the following rules for multiplication on $\Fn$:
\begin{eqnarray*}
e_0 = 1,\qquad&&\cr 
e_j^2 = -1,\qquad&&{\rm for}\quad 1 \leq j \leq n\cr 
e_je_k =
-e_ke_j = e_{\{j,k\}},\qquad&&{\rm for}\quad 1 \leq j < k  \leq n\cr
e_{j_1}e_{j_2}\cdots e_{j_s} = e_S,\qquad&&
{\rm if}\quad 1 \leq j_1 < j_2 < \dots<
j_s
\leq n\\
&&\qquad\qquad{\rm and}\quad S=\{j_1,\dots,j_s\}.
\end{eqnarray*}
\glossary{$e_S$}
Here the identifications $e_0 =
e_\emptyset$ and $e_j = e_{\{j\}}$ for $1 \leq j \leq n$ have been made.

Suppose that $m \leq n$ are positive integers. The vector space $\R^m$ is identified with a
subspace of $\Fn$ by virtue of the embedding $(x_1,\dots,x_m)\longmapsto \sum_{j=1}^mx_je_j$. On
writing the coordinates of $x\in\R^{n+1}$ as $x = (x_0,x_1,\dots,x_n)$, the space $\R^{n+1}$ is
identified with a subspace of $\Fn$ with the embedding $(x_0,x_1,\dots,x_n)\longmapsto
\sum_{j=0}^nx_je_j$.

The product of two elements $u = \sum_S u_S e_S$  and $v = \sum_S v_S  e_S,
v_S
\in \F$ with coefficients $u_S \in \F$ and $v_S \in \F$ is $uv = \sum_{S,R}u_S v_R e_S e_R$. According
to the rules for multiplication,
$e_Se_R$ is $\pm1$ times a basis vector of $\Fn$. The {\it scalar part} of $u = \sum_S u_S e_S,
u_S
\in \F$ is the term $u_\emptyset$, also denoted as $u_0$.

The Clifford algebras $\R_{(0)},\R_{(1)}$ and $\R_{(2)}$ are the real, 
\glossary{$\R_{(0)}$}\glossary{$\R_{(1)}$}\glossary{$\R_{(2)}$}
complex numbers and the
quaternions\index{quaternions}, respectively. In the case of $\R_{(1)}$, the vector $e_1$ is identified with
$i$ and for $\R_{(2)}$, the basis vectors $e_1, e_2, e_1e_2$ are 
identified with $i,j,k$ respectively. 
 
The conjugate $\overline{{e_S}}$ of a basis element $e_S$ is defined so that $e_S\overline {e_S}
=
\overline {e_S}e_S = 1$. Denote the complex conjugate of a number $c \in \F$ by
$\overline c$. Then the operation of {\it conjugation\index{Clifford!conjugation}} 
$u\longmapsto\overline{u}$\glossary{$\overline{u}$} defined by
$\overline u = \sum_S \overline{u_S}\,\overline{e_S}$ for every $u = \sum_S u_S e_S, u_S \in
\F$ is an involution of the Clifford algebra $\Fn$ and $\overline{vu}=\overline{u}\,
\overline{v}$ for all elements $u$ and $v$ of $\Fn$. Because $e_j^2=-1$, the conjugate 
$\overline{e_j}$ of $e_j$ is $-e_j$.
An inner product is defined on $\Fn$ by
the formula $(u,v) = [u\overline{v}]_0 = \sum u_S\overline{v_S}$ for every $u = \sum_S u_S e_S$
and $v =
\sum_S v_S e_S$ belonging to $\Fn$. The corresponding norm is written as $|\cdot|$.

For a Banach space $X$, the tensor product $X_{(n)} : = X\otimes\Cn$ denotes the
module of all sums $u = \sum_S u_S e_S$ with coefficients $u_S\in X$ endowed with the norm
$$\|u\|_{X_{(n)}} = \left(\sum_S \|u_S\|_X^2\right)^\frac12.$$
 The left product
$\lambda u$ and right product $ u\lambda$ are defined in the obvious way for all $\lambda\in \Cn$.
The space $\cL_{(n)}(X_{(n)})$ of right-module homomorphisms is identified with $\cL(X)_{(n)}$ by writing
$$\left(\sum_S T_S e_S\right)\left(\sum_{S'} u_{S'} e_{S'}\right) = \sum_S (T_Su_{S'}) e_Se_{S'},$$
so that $T(u\lambda) = (Tu)\lambda$ for all $u\in X_{(n)}$, $\lambda \in \Cn$ and $T = \cL(X)_{(n)}$.
The linear subspace $\{Te_0:T\in\cL(X)\}$ of $\cL(X)_{(n)}\equiv \cL_{(n)}(X_{(n)})$ is identified with $\cL(X)$.
The norm induced on $\cL_{(n)}(X_{(n)}) \equiv \cL(X)_{(n)}$ by the norm of $X_{(n)}$ is given by
$$\big\|\sum_S T_S e_S\big\|_{\cL_{(n)}(X_{(n)})} = \left(\sum_S \|T_S\|_{\cL(X)}^2\right)^\frac12.$$
For an $n$-tuple $\bA$ of bounded linear operators on $X$, it is convenient to write $\|\bA\|$  for the norm
$\| A_1e_1+\dots+A_ne_n\|_{\cL_{(n)}(X_{(n)})} = \big(\sum_j \|A_j\|_{\cL(X)}^2\big)^\frac12$.

Because $x = (x_0,x_1,\dots,x_n)\in\R^{n+1}$ is identified
with the element $\sum_{j=0}^n x_je_j$ of $\R_{(n)}$, the conjugate $\overline x$
of $x$ in $\R_{(n)}$ is $x_0e_0 -x_1e_1-\cdots -x_ne_n$. A useful feature of Clifford algebras
is that a nonzero vector $x\in \R^{n+1}$ has an inverse $x^{-1}$ in the algebra
$\R_{(n)}$ (the {\it Kelvin inverse\index{Kelvin inverse}}) given by 
$$
x^{-1} = \frac{\overline{x}}{|x|^2} = \frac{x_0e_0-x_1e_1-\cdots -x_ne_n}
{x_0^2+x_1^2+\cdots +x_n^2}.
$$
The vector $x = (x_0,x_1,\dots,x_n)\in\R^{n+1}$ will often be written as  $x = x_0e_0 +\bs x$ with
$\bs x= (x_1,\dots,x_n)\in\R^n$.

Because $\Cn$ is an algebra over the complex numbers, the \textit{spectrum} $\s(i \bs x)$
of the vector $i\bs x$ with $\bs x \in\R^n$ is the set of all $\lambda\in \C$ for which
$\lambda e_0-i\bs x$ is \textit{not invertible} in $\Cn$. The formula
$$(\lambda e_0-i\bs x)^{-1} = \frac{\lambda e_0+i\bs x}{\lambda^2-|\bs x|^2},\quad 
\lambda\neq \pm|\bs x|,$$
ensures that $\s(i \bs x) =\{\pm|\bs x|\}$ if $\bs x\neq 0$ and $\s(0) = \{0\}$.
For $\bs x \neq 0$, the spectral representaion
$$i\bs x = |\bs x|\xchi_+(\bs x) + (-|\bs x|)\xchi_-(\bs x)$$
holds with respect to the spectral idempotents
$$\xchi_\pm(\bs x) =\frac12\left(e_0\pm i\frac{\bs x}{|\bs x|}\right).$$
The vector $i\bs x$ is actually selfadjoint with respect to the inner product of $\Cn$ defined above.
For every function $f:\{\pm|\bs x|\}\to \C$ there is an element
\begin{equation}\label{eqn:SR}
f(i\bs x)=f(|\bs x|)\xchi_+(\bs x) + f(-|\bs x|)\xchi_-(\bs x) 
\end{equation}
of $\Cn$ associated with $f$ by the functional calculus for selfadjoint operators.
For a polynomial $p(z) = a_0 +a_1z+\dots+a_kz^k$, the expression
$$p(i\bs x) = a_0e_0 +a_1(i\bs x)+\dots+a_k(i\bs x)^k$$
expected  in $\Cn$ is obtained. The identities $\xchi_{\R_\pm}(i\bs x) = \xchi_{\pm}(\bs x)$
hold for the characteristic functions $\xchi_{\R_\pm}$ of the half lines $\R_+=\{t > 0\}$
and $\R_-=\{t < 0\}$.

Given an $n$-tuple of operators $\bs A=(A_1,\dots,A_n)$ acting on a Hilbert space $\cH$,
the expression $\cA =\sum_{j=1}^n e_jA_j$ acts on $\Hn:= \cH\otimes\Cn$ via the formula
$$\cA u =  \sum_{j=1}^n\sum_{S} (e_je_S)(A_ju_S),\quad u = \sum_{S} u_Se_S.$$
The coefficients $u_S$ are elements of the Hilbert space $\cH$ and $u_S\otimes e_S$ is written simply as $u_Se_S$ for all $S\subseteq \{1,\dots,n\}$
.

Using Fourier theory, the functional calculus for the selfadjoint differential operator
$$\cD = \sum_{j=1}^n e_j\frac{\partial}{\partial x_j}$$
acting in the Hilbert space $L^2_{(n)}(\R^n) := L^2(\R^n) \otimes\Cn$ can be calculated explicitly. 

If $\hat u(\xi) = \int_{\R^n}
e^{-i\langle x,\xi\rangle}u(x)\,dx$ denotes the Fourier transform of $u\in L^1(\R^n)$,
then according to the Fourier-Plancherel Theorem,
the linear map $u\longmapsto (2\pi)^{-n/2}\hat u$, $u\in L^1\cap L^2(\R^n)$, extends to an isometry of $L^2(\R^n)$. For each $j=1,\dots,n$, the selfadjoint operator $\frac1i\frac{\partial}{\partial x_j}$
defined in $L^2(\R^n)$ satisfies
$$\left(\frac1i\frac{\partial u}{\partial x_j}\right)\widehat{\phantom{\big|}}(\xi)=\xi_j\hat u(\xi)$$
almost everywhere for each $u\in L^2(\R^n)$ in its domain. Furthermore, for any bounded measurable function $\varphi$ defined on $\R^n$, the operator $\varphi\left(\frac1i\frac{\partial}{\partial x_1},\dots,\frac1i\frac{\partial}{\partial x_n}\right)$ satisfies
$$\left(\varphi\left(\frac1i\frac{\partial}{\partial x_1},\dots,\frac1i\frac{\partial}{\partial x_n}\right)u\right)\widehat{\phantom{\big|}}(\xi)=\varphi(\xi)\hat u(\xi) $$
almost everywhere for $u\in L^2(\R^n)$.
Similarly,
$$\big(f(\cD)u)\big)\widehat{\phantom{|}}(\xi) = f(i\xi)\hat u(\xi),\quad u\in L^2_{(n)}(\R^n),
\ \xi\in \R^n,$$
is valid for any bounded measurable function $f:\R\to \C$ with the understanding that
for each $\xi\in\R^n$, the vector $f(i\xi)\in \R^n$ is given by the functional calculus (\ref{eqn:SR}) of the selfadjoint element $i\xi$ of the Clifford algebra $\Cn$.

What is usually called {\it Clifford analysis\index{Clifford!analysis|)}} is the
study of functions of finitely many real variables, 
which take values in a Clifford algebra\index{Clifford!algebra|)},
and which satisfy higher dimensional analogues of the 
Cauchy-Riemann equations\index{Cauchy-Riemann equations|(}. 

It is worthwhile to spell out the direction this analogy takes. The
Cauchy-Riemann equations for a complex valued function $f$ defined 
in an open subset of the complex plane may be represented as $\overline \pd f = 0$
for the operator \glossary{$\overline \pd$}
$$\overline \pd =\frac{\pd}{\pd x}+i\frac{\pd}{\pd y},\quad z=x+iy\in\C.$$
The fundamental solution $E$ of the operator $\overline \pd$ is the solution
in the sense of Schwartz distributions of the equation $\overline \pd E = \delta_0$
for the unit point mass $\delta_0$ at zero. Then 
$$E(z) = \frac{1}{2\pi }\,\frac{1}{z} = \frac{1}{2\pi }\,\frac{\overline z}{|z|^2},\quad \hbox{for }
z=x+iy\in\C\setminus\{0\}.$$
A function $f$ satisfying $\overline \pd f = 0$ in a neighbourhood of a simple closed
contour $C$ together with its interior can be represented as
$$f(z) = \frac1{i}\int_C E(\zeta -z)f(\z)\,d\z = \int_C E(\zeta -z)\bs n(\z)f(\z)\,d|\z|$$
at all points $z$ inside $C$. Here $\bs n(\z)$ is the outward unit normal at $\z\in\C$,
$d|\z|$ is arclength measure so that $i\bs n(\z)d|\z| = d\z$.
The higher dimensional analogue for functions taking values in
a Clifford algebra is as follows.

A function $f:U \to \Fn$ defined in an
open subset $U$ of $\R^{n+1}$ has a unique representation $f = \sum_Sf_Se_S$ in  terms of
$\F$-valued functions $f_S$, $S \subseteq \{1,\dots,n\}$ in the sense that $f(x) =
\sum_Sf_S(x)e_S$ for all $x\in U$. Then $f$ is continuous, differentiable and so on, in the
normed space $\Fn$, if and only if  for all finite subsets $S$ of $\{1,\dots,n\}$, its scalar
component functions $f_S$ have the corresponding property. Let $\pd_j$ be the operator of
differentiation of a scalar function in the
$j$'th coordinate in $\R^{n+1}$ -- the coordinates of $x\in\R^{n+1}$ are written as $x =
(x_0,x_1,\dots,x_n)$. For a continuously differentiable function $f:U \to \Fn$ with $f =
\sum_Sf_Se_S$ defined in an
open subset $U$ of $\R^{n+1}$, the functions $Df$ and $fD$ are defined by 
\begin{align*}
Df &= \sum_S\left((\pd_0
f_S)e_S+\sum_{j=1}^n(\pd_j f_S)e_je_S\right)\\
fD &= \sum_S\left((\pd_0
f_S)e_S+\sum_{j=1}^n(\pd_j f_S)e_Se_j\right).
\end{align*} 
Similarly, the conjugate operator is defined by
$$\overline D =e_0\pd_0 - \sum_{j=1}^ne_j\pd_j$$
so that the Laplace operator $\Delta =\pd_0^2+\dots+\pd_{n}^2$ in $\R^{n+1}$ has factorizations
$$\Delta = D\overline D=\overline D D.$$

Now suppose that $f$ is an $\Fn$-valued, continuously differentiable function defined in an
open subset $U$ of $\R^{n+1}$. Then $f$ is said to be 
{\it left monogenic\index{function!left monogenic|(}} in
$U$ if
$Df(x) = 0$ for all $x\in U$ and {\it right monogenic\index{function!right monogenic}} in $U$ if
$fD(x) = 0$ for all $x\in U$.

The following result connects $\R^{n+1}$-valued monogenic functions with systems of conjugate harmonic functions.

\begin{prp} Let $F =u_0e_0-\sum_{j=1}^nu_je_j$ be an $\R^{n+1}$-valued function defined on an open subset $\Omega$ of $\R^{n+1}$. Conditions (1)-(4) below are equivalent in $\Omega$.
\begin{itemize}
\item[(1)] The $(n+1)$-tuple $U = (u_j)_{j=0}^n$ is a system of conjugate harmonic functions in $\Omega$, that is, $U$ satisfies the generalised Cauchy-Riemann equations $\hbox{\rm div}\, U=0$ and $\hbox{\rm curl}\, U=0$.
\item[(2)] $F$ is left monogenic.
\item[(3)] $F$ is right monogenic.
\item[(4)] The 1-form $\omega := u_0dx_0 -u_1dx_1-\dots-u_ndx_n$ satisfies $d\om=0$ and $d^*\om=0$, where $d$ and $d^*$ are the exterior differential operator and its formal transpose, respectively.
\item[(5)] In the case that $\Omega$ is simply connected, then the above conditions are equivalent to
the existence of a real valued harmonic function $v$ in $\Omega$ such that $U=\hbox{\rm grad}\,v$, so that $F=\overline Dv$.
\end{itemize}
\end{prp}

For each $x\in \R^{n+1}$, the function $G(\var,x)$ defined by 
\begin{equation}\label{eqn:3.1}
G(\omega,x) = G_\omega(x) = \frac{1}{ \Sigma_n}\, \frac{\overline{\omega-x}}{ |\omega-x|^{n+1}}
\end{equation}
for every $\omega \neq x$ is both left and right monogenic as a function of $\omega$. Here the volume
${2\pi^{\frac{n+1}{2}}/\Gamma\left(\frac{n+1}{2}\right)}$ of the unit
$n$-sphere in $\R^{n+1}$ has been denoted by $\Sigma_n$ and we have used the identification of $\R^{n+1}$
with a subspace of $\R_{(n)}$ mentioned earlier.

The function $G(\var,x)$, $x\in\R^{n+1}$ plays the role in Clifford analysis
of a {\it Cauchy kernel\index{Cauchy!kernel}}. Rewriting $G(\omega,x)$ as $E(\omega - x)$ for all $\omega\neq x$
in $\R^n$, it follows that the $\R^{n+1}$-valued function 
$$E(x) = \frac{1}{ \Sigma_n}\, \frac{\overline x}{|x|^{n+1}}$$
defined for all $x\ne 0$ belonging to $\R^{n+1}$ is the fundamental solution of
the operator $D$, that is, $DE = \delta_0e_0$ in the sense of Schwartz distributions, because
$E=\overline D\G_{n+1} = \G_{n+1}\overline D $ for the fundamental solution
$$\G_{n+1}(x)=\left\{\begin{array}{rl}\displaystyle -\frac1{(n-1)\Sigma_n}\frac1{|x|^{n-1}},&x\neq0,\ n\ge 2\\
&\\
\displaystyle\frac1{2\pi}\log|x|,&x\neq0,\ n = 1,
\end{array}\right.
$$
of the Laplace operator $\Delta$ in $\R^{n+1}$.
Then a function satisfying $Df = 0$ in an open set can be retrieved from a
surface integral involving $E$.

Suppose that
$\Omega\subset\R^{n+1}$ is a bounded open set with smooth boundary $\pd \Omega$ and exterior unit normal
$\bs n(\omega)$ defined for  all $\omega\in \pd \Omega$. For any left monogenic function $f$ defined in a
neighbourhood
$U$ of $\overline \Omega$, the Cauchy integral formula
\begin{equation}\label{eqn:CIF}
\int_{\pd \Omega}G(\omega,x)\bs n(\omega)f(\omega)\,d\mu(\omega) =
\left\{\begin{array}{rl}  f(x),&{\rm if}\quad
x\in\Omega;\cr 
0,&{\rm if}\quad x\in U\setminus\overline\Omega.
\end{array}\right.
\end{equation}
is valid. Here
$\mu$ is the surface measure of $\pd
\Omega$. The result is proved in \cite[Corollary 9.6]{BDS}
by appealing to  by Stoke's theorem. If $g$ is right monogenic in $U$ then
$\int_{\pd \Omega} g(\omega)\bs n(\omega)f(\omega)\,d\mu(\omega) = 0$ \cite[Corollary 9.3]{BDS}.
The manifold $\pd \Omega$ integral (\ref{eqn:CIF}) can then be deformed across any set not containing $x$.

In terms of differential forms, the $\F^{n+1}$-valued $n$-form 
$$\eta = \sum_{j=0}^n (-1)^j e_j\,dx_0\wedge\dots\wedge \widehat{dx_j}\wedge\dots\wedge dx_n
\quad(\,\widehat{\phantom{a}} \equiv\hbox{omitted}),$$
is defined on $\R^{n+1}$ and its pullback to the orientable $n$-dimensional manifold $\pd \Omega$ is denoted by the same symbol. Then for any continuous $\Fn$-valued functions $u,v$ on $\pd\Omega$, the equality
$$\int_{\pd \Omega} u \eta v = \int_{\pd \Omega} u(\omega)\bs n(\omega)v(\omega)\,d\mu(\omega)$$
holds. Now suppose that $f,g$ are any continuously differentiable $\Fn$-valued functions defined on $U$.
The differential of $g \eta f$ is equal to $\big((gD)f + g(Df)\big)\,dx_0\wedge\dots\wedge dx_n$, so Stokes' Theorem gives
$$\int_{\Omega}\big((gD)f + g(Df)\big)\,d\lambda = \int_{\pd \Omega} g \eta f $$
with respect to Lebesgue measure $\lambda$ on $\R^{n+1}$. The Cauchy integral formula
(\ref{eqn:CIF}) follows by shrinking $\pd \Omega$ to a sphere about $x\in \Omega$, see \cite[Corollary 9.6]{BDS}.

 \begin{xmp}\label{xmp:2.1}
For the case $n = 1$, the Clifford algebra $\R_{(1)}$ is identified with $\C$. A continuously
differentiable function $f:U\to\R_{(1)}$ defined in an open subset $U$ of $\R^2$ satisfies $Df
= 0$ in $U$ if and only if it satisfies the 
Cauchy-Riemann equations\index{Cauchy-Riemann equations|)} 
$\overline\pd f = 0$ in $U$. For each
$x,\omega\in\R^2$,
$x\neq \omega$, the formula
$$G(\omega,x)= \frac{1}{2\pi}\ \frac{1}{ \omega-x}$$
holds. The inverse is taken in $\C$. 
As indicated above, the tangent at the point
$\z(t)$ of the portion $\{\z(s): a < s < b\}$ of a positively oriented   rectifiable curve $C$
is
$i$ times the normal $\bs n(\z(t))$ at
$\z(t)$, so the equality $d\z = i.\bs n(\z)\,d|\z|$ shows that (\ref{eqn:CIF}) is the Cauchy integral
formula for a simple closed contour $C$ bounding a region $\Omega$.
\end{xmp}

\section{The Cauchy kernel}
Armed with the Cauchy integral formula (\ref{eqn:CIF}) for monogenic functions, formula
(\ref{eqn:fc}) is established for the $n$-tuple $\bs A =
(A_1,\dots,A_n)$ of bounded linear operators on
a Banach space $X$ by  substituting the $n$-tuple $\bs A$ for the vector $y\in\R^n$.
If $n$ is odd and $\bs A$ is a commutative $n$-tuple, that is, $A_jA_k = A_kA_j$ for $j,k=1,\dots,n$,
and each operator $A_j$ has real spectrum $\s(A_j)\subset \R$ for $j=1,\dots,n$, then for suitable $\omega\in\R^{n+1}$, the expression
\begin{equation}\label{eqn:3.19}
G_\om(\bs A)=\frac{1}{\Sigma_n }|\omega I- \bs A|^{-n-1}(\overline{\omega I- \bs A})
\end{equation}
makes sense as an element of
$\cL_{(n)}(X_{(n)})$. For an even integer  $m$
$$|\omega I-\bs A|^{-m} = \bigg(\big(\omega_0^2I
+\sum_{j=1}^n(\omega_jI-A_j)^2\big)^{-1}\bigg)^{m/2}$$  and
$\overline{\omega I- \bs  A}= \omega_0I-\sum_{j=1}^n(\omega_jI-A_j)e_j$ for $\om=\om_0e_0+\bs\om \in \R^{n+1}$, 
$\bs\om =\sum_{j=1}^n\om_je_j$.

An appeal to the Spectral Mapping Theorem shows that the operator 
$$\omega_0^2I +\sum_{j=1}^n(\omega_jI-A_j)^2$$ 
is invertible in $\cL(X)$ for each $\omega_0 \neq 0$, so the $\cL_{(n)}(X_{(n)})$-valued function
$\om\longmapsto G_\om(\bs A)$ is defined on the set $\R^{n+1}\setminus (\{0\}\times \g(\bs A))$ with 
$$\g(\bs A) = \left\{(\om_1,\dots,\om_n): \sum_{j=1}^n(\omega_jI-A_j)^2
\hbox{ is not invertible in }\cL(X)\,\right\}.$$
Off $\g(\bs A)$, the function $\om \longmapsto G_\om(\bs A)$ is left and right monogenic
and formula (\ref{eqn:fc}) defines a functional calculus $f\longmapsto f(\bs A)$ which coincides with Taylor's functional calculus $\tilde f\longmapsto \tilde f(\bs A)$. Any left monogenic function
$f$ defined in a neighbourhood $U$ of $\g(\bs A)$ in $\R^{n+1}$ has a holomorphic counterpart $\tilde f$
defined in a neighborhood $\tilde U$ of  $\g(\bs A)$ in $\C^{n}$ by taking the power series expansion about points of
$U\cap (\{0\}\times \R^n)$ \cite{BDS}. The left monogenic function $f$ is referred to as
the \textit{Cauchy-Kowaleski extension} of $\tilde f\restriction (\tilde U\cap\R^n)$ to $\R^{n+1}$ \cite{BDS}.

In the case of even $n =2,4,\dots$, the operator $|\omega I- \bs A|^{-n-1}$ needs to be defined suitably. The direct formulation employs Taylor's functional calculus, but by using the 
\textit{plane wave decomposition} of the Cauchy kernel (see Section 4.2), the case of even $n$ and noncommuting operators can be treated simultaneously.

For an $n$-tuple $\bs A =
(A_1,\dots,A_n)$ of commuting bounded linear operators on
a Banach space $X$ with real spectra, the nonempty compact subset  $\g(\bs A)$ of $\R^n$
coincides with Taylor's joint spectrum defined in terms of the Koszul complex \cite[Definition III.6.4]{Vascu}. 

In general, the symbol $\g(\bs A)$ is used to denote the set of points of $\R^{n+1}$ in the complement of the domain where the Cauchy kernel $\om\longmapsto G_\om(\bs A)$ is defined and monogenic in the Banach module $\cL_{(n)}(X_{(n)})$. For a single operator $A$, its spectrum $\s(A)$ is precisely the set of singularities
of the Cauchy kernel or resolvent $\lambda\longmapsto (\lambda I-A)^{-1}$ for $\lambda\in \C$, that is, the set of $\lambda\in \C$ for which $\lambda I-A$ is not invertible in $\cL(X)$.

A commuting $n$-tuple $\bs A =
(A_1,\dots,A_n)$ of bounded selfadjoint operators on a Hilbert space $H$ has a ready-made
functional calculus given by formula (\ref{eqn:fc}). The support of the joint spectral measure $P_{\bs A}$
is naturally interpreted as the joint spectrum $\s(\bs A)$ of $\bs A$. The observation that
$\s(\bs A)$ is actually the Gelfand spectrum of the commutative $C^*$-algebra generated by $\bs A$ lends
credence to the interpretation of $\s(\bs A)$ as the joint spectrum. By setting
$G_\om(\bs A) = \int_{\s(\bs A)}G_\om(\lambda)\,dP_{\bs A}(\lambda)$ for all $\om\in \R^{n+1}\setminus (\{0\}
\times \s(\bs A))$, it is easy to check by the vector valued version of Fubini's Theorem, that
with the assumptions of formula (\ref{eqn:fc}) below for the left monogenic function $f:U\to\Cn$ and the open set $\Omega$, the equalities
\begin{eqnarray*}
\int_{\partial \Omega}G_\om(\bs A)\bs n(\om)f(\om)\,d\mu(\om) &=& 
\int_{\partial \Omega}\left(\int_{\s(\bs A)}G_\om(\lambda)\,dP_{\bs A}(\lambda)\right)\bs n(\om)f(\om)\,d\mu(\om)\\
&=& \int_{\s(\bs A)}\left(\int_{\partial \Omega}G_\om(\lambda)\bs n(\om)f(\om)\,d\mu(\om)\right)\,dP_{\bs A}(\lambda)\\
&=& \int_{\s(\bs A)} f(\lambda)\,dP_{\bs A}(\lambda)\\
&=&f(\bs A)
\end{eqnarray*}
hold. Furthermore $\g(\bs A) = \s(\bs A)$.

\subsection{The Weyl calculus}
There is an operator valued distribution $\cW_{\bs A}$ on $\R^n$ associated with any $n$-tuple $\bs A =
(A_1,\dots,A_n)$ of selfadjoint operators on a Hilbert space $H$. Now the operators $A_1,\dots,A_n$
need not commute with each other. The distribution $\cW_{\bs A}$ is a substitute for the 
joint spectral measure $P_{\bs A}$. If $\bs A$ is a commuting $n$-tuple of bounded selfadjoint
operators, then $\cW_{\bs A}:f\longmapsto \int_{\s(\bs A)}f\,dP_{\bs A}$ for all smooth functions
defined in a neighbourhood of $\s(\bs A)$ in $\R^n$.

Suppose that $T:C^\infty(\R^n)\to\cL(X)$ is an operator valued distribution with compact support supp$(T)$ acting on a Banach space $X$. For any smooth $\Cn$-valued function function $f=\sum_Sf_Se_S$ defined in a neighborhood of supp$(T)$ in $\R^n$, the element $Tf$ of $\cL_{(n)}(X_{(n)})$ is defined by
$Tf=\sum_Se_ST(f_S)$

The Cauchy integral formula (\ref{eqn:CIF}) may be viewed as an equality
$$f = \int_{\pd \Omega}G_\omega\bs n(\omega)f(\omega)\,d\mu(\omega)$$
between smooth $\Cn$-valued functions defined in a neighbourhood $U\cap \R^n$ of the support of $T$
when $f$ is left monogenic on $U$, supp$(T)\subset \Omega$ and $\overline\Omega\subset U$, so that
\begin{eqnarray*}
Tf &=& T\int_{\pd \Omega}G_\omega\bs n(\omega)f(\omega)\,d\mu(\omega)\\
&=& \int_{\pd \Omega}T(G_\omega)\bs n(\omega)f(\omega)\,d\mu(\omega).
\end{eqnarray*}
The last inequality is a property of the Bochner integral for functions with
values in the Fr\'echet space $C^\infty_{(n)}(\R^n)$. The $\cL_{(n)}(X_{(n)})$-valued function $\om\longmapsto T(G_\omega)$
is left and right monogenic in $\R^{n+1}$ away from supp$(T)$. If $\varphi$ is a smooth function with compact support in a neighbourhood of supp$(T)$, $x\in X$ and $\xi\in X'$, then according to \cite[Theorem 27.7]{BDS},
$$\langle (T\varphi)x,\xi\rangle = \lim_{t\to0+}\int_{\R^n}\langle \left(T(G_{u+te_0})-T(G_{u-te_0})\right) x,\xi\rangle\varphi(u)\,du,$$
which may be compared with equation (\ref{eqn:Brem}).

For any $n$-tuple $\bs A =
(A_1,\dots,A_n)$ of bounded selfadjoint operators on a Hilbert space $H$, the Weyl functional calculus
is the $\cL(H)$-valued distribution
\begin{equation}\label{eqn:Weyl}
\cW_{\bs A} = \frac1{(2\pi)^n}\left(e^{i\langle \xi, \bs A\rangle}\right)\hat{}\phantom{|}.
\end{equation}
The operator $\langle \xi, \bs A\rangle= \langle \bs A,\xi\rangle= \sum_{j=1}^n \xi_jA_j$ is selfadjoint for each $\xi\in\R^n$ and the Fourier transform is taken with respect $\xi$ in the sense of distributions. If $\bs A$ is a commuting system, then the distribution $\cW_{\bs A}$ is integration with respect to the joint spectral measure $P_\bs A$.

Setting
$G_\om(\bs A) = \cW_{\bs A}(G_\omega\restriction\R^n)$ for $\om\in \R^{n+1}\setminus(\{0\}\times\hbox{supp}(\cW_{\bs A}))$, the equality $f(\bs A) = \cW_{\bs A}(f\restriction\R^n)$ holds
for the element $f(\bs A)$ of $\cL_{(n)}(H_{(n)})$ defined by formula (\ref{eqn:fc}) below and $\g(\bs A)= \hbox{supp}(\cW_{\bs A})$. E. Nelson \cite{Nel} has identified $\g(\bs A)$ with the Gelfand spectrum
of a certain commutative Banach algebra.

\begin{xmp}\label{xmp:3.1.1} Let $n = 3$ and consider the simplest noncommuting example of the Pauli
matrices,  
\begin{equation}\label{eq:Paulim}
\sigma_{1} = \left(\begin{array}{cc} 0&1\\ 1&0\end{array}\right),\qquad \sigma_{2} =
\left(\begin{array}{cc} 0&{- i}\\ {i}&0\end{array}\right),\qquad \sigma_{3} = \left(\begin{array}{cc} 1&0\\ 0&{-1}\end{array}\right),
\end{equation}
viewed as linear transformations acting on $H = \C^2$. Set $\bs\sigma = (\sigma_1,\sigma_2,\sigma_3)$. Then
$\langle \xi, \bs \s\rangle^2=|\xi|^2 I$, so the exponential series gives
$$e^{it\langle \xi, \bs \s\rangle} =\cos (t|\xi|) I + i\langle \xi, \bs \s\rangle\frac{\sin(t|\xi|)}{|\xi|},\quad t\in\R.$$
Because $(2\pi)^{-3}({\sin(t|\xi|)}/|\xi|)\,\hat{} = t\mu_t$, with $\mu_t$ the unit surface measure on the sphere
$tS^{2}$ of radius $t >0$ centred at zero in $\R^3$ and 
$$\cos(t|\xi|)= \frac{d}{ds}\,\frac{\sin(s|\xi|)}{|\xi|}\bigg|_{s=t},$$
for each $f\in C^\infty(\R^3)$, the matrix $\cW_{t\bs \s}(f)$ is
given by 
\begin{equation}\label{eq:Pauli}
\cW_{t\bs \s}(f) = I\int_{tS^{2}}\left(f + \bs n\cdot\nabla f\right)\,d\mu_t +
t\int_{tS^{2}}{\bs\sigma}\cdot\nabla f\,d\mu_t,\quad t> 0.
\end{equation}
 Here  $\bs n(x)$ is
the outward unit normal at $x\in S^{2}$. Thus, $\supp\left(\cW_{\bs \s}\right) = S^{2}$. 

For all $\omega =\om_0+\bs\om \in\R^4$ such that $\bs\om \notin S^{2}\subset \R^3$, the Cauchy kernel 
$G_\omega(\bs\s)\in \cL(\C^2)_{(3)}$
is given by  
$$G_\omega(\bs\s)=\cW_{\bs \s}(G_\omega ) = I\int_{S^{2}}\left(G_\omega  + \bs n\cdot\nabla G_\omega \right)\, d\mu +
\int_{S^{2}}{\bs\sigma}\cdot\nabla G_\omega  d\mu.$$
\end{xmp}

By the operator valued version of the Paley-Wiener Theorem (see \cite[Theorem 7.23]{Rud} for the scalar version), the distribution (\ref{eqn:Weyl}) exists
for $n$-tuple $\bs A =
(A_1,\dots,A_n)$ of bounded linear operators on a Banach space $X$ provided that
the exponential growth estimate
\begin{equation}\label{eqn:exp}
\left\|e^{i\langle \z, \bs A\rangle}\right\|_{\cL(X)} \le C(1+|\xi|)^se^{r|\eta|},\quad \z=\xi+i\eta,\ \xi,\eta\in \R^n,
\end{equation}
for $\langle \z, \bs A\rangle=\sum_{j=1}^n \z_jA_j$ holds, for some positive numbers $C,s,r$ independent of $\z\in\C^n$. Then supp$(\cW_{\bs A})$ is contained in the ball of radius $r > 0$ centred at zero in $\R^n$.
The estimate (\ref{eqn:exp}) holds if $\bs A$ is finite system of simultaneously triangularisable matrices with real eigenvalues \cite[Theorem 5.10]{J}.

The Weyl functional calculus $\cW_{\bs A}$ has the property that operator products are 
\textit{symmetrically} ordered. For an $n$-tuple $\bs A =
(A_1,\dots,A_n)$ of bounded selfadjoint operators on a Hilbert space  $H$, other choices of operator ordering define an operator valued distribution $\cF_{\bs A,\bs \mu}$. The weighting for operator products is determined by an $n$-tuple $\bs\mu=(\mu_1,\dots,\mu_n)$ of continuous Borel probability measures on $[0,1]$, see \cite[Chapter 7]{J}.

The collection of single operators $A$ satisfying the bound (\ref{eqn:exp}) is called the class of \textit{generalized scalar operators} and these have been extensively studied \cite{CF}.

\subsection{Plane wave decomposition of the Cauchy kernel}  The exponential growth estimate (\ref{eqn:exp}) leads to a $C^\infty$-functional calculus $\cW_\bs A$, so it is desirable to have a condition weaker than (\ref{eqn:exp}) for which the Cauchy kernel can be defined in a way that agrees with the preceding definition.

It turns out that it suffices for $\bA$ to be \textit{hyperbolic} to make sense of the Cauchy kernel $G_\om(\bs A)$ in the
Cauchy integral formula (\ref{eqn:fc}). For \textit{matrices}, this is equivalent to the bound (\ref{eqn:exp}), because 
as noted in the Introduction it says that
$$I\frac{\pd}{\pd t} + \sum_{j=1}^nA_j\frac{\pd}{\pd x_j} $$
is a \textit{hyperbolic differential operator} on $\R^{n+1}$.

The key to constructing the Cauchy kernel $G_\om(\bs A)$  for a hyperbolic $n$-tuple $\bs A$ of bounded linear operators  is the plane wave decomposition of the fundamental solution 
$$E:x\longmapsto \frac{1}{\Sigma_n}\frac{\overline x}{|x|^{n+1}},\quad x\in \R^{n+1}\setminus\{0\},$$ 
of the generalised Cauchy-Riemann operator $D=\sum_{j=0}^ne_j\pd_j$. The plane wave decomposition of $E$ was first
given by F. Sommen \cite{Som} and is most simply realised with the proof of Li, McIntosh, Qian \cite{LMQ} using
Fourier analysis. 
The unit hypersphere $S^{n-1}$ in $\R^n$ is the set $\{s\in\R^n: |s|=1\,\}$.

The Fourier transform of $u\in L^1\cap L^2(\R^n)$ is $\hat u(\xi)=\int_{\R^n}e^{-i\langle x,\xi\rangle}u(x)\,dx$
and the inverse map $u(x) = (2\pi)^{-n}\int_{\R^n}e^{i\langle x,\xi\rangle}\hat u(\xi)\,d\xi$ recovers $u$ from $\hat u$ when $\hat u\in L^1\cap L^2(\R^n)$. For any $\xi\in \R^n$, the linear function $\bs x\longmapsto i\langle \bs x,\xi\rangle$, $\bs x\in \R^n$, defined in $\R^n$ extends monogenically to $\R^{n+1}$ to the function
$$x\longmapsto i\langle \bs x,\xi\rangle e_0 -i\bs x x_0,\quad x=x_0e+\bs x, x_0\in\R,\ \bs x\in \R^{n}.$$
According to the functional calculus for the selfadjoint element $i\bs x$ of $\Cn$,
the unique monogenic extension of the function $\bs x \longmapsto e^{i\langle \bs x,\xi\rangle}$, $\bs x\in\R^n$, is given by
\begin{eqnarray*}
\exp(i\langle \bs x,\xi\rangle e_0 -i\bs x x_0) &=& e^{i\langle \bs x,\xi\rangle-|\bs x|x_0}\xchi_+(\bs x) +
e^{i\langle \bs x,\xi\rangle+|\bs x|x_0}\xchi_-(\bs x)\\
&=& e_+(x,\xi)+e_-(x,\xi)
\end{eqnarray*}
for $x=x_0e+\bs x$ with  $x_0\in\R$ and $\bs x\in \R^{n}$.
\begin{thm} Let $x = x_0e_0+\bs x$ be an element of $\R^{n+1}$
with $\bs x\in\R^n$. If $x_0>0$, then
\begin{equation}\label{eqn:PW1}
E(x) = 
\frac{(n-1)!}{ 2}\left(\frac{i}{
2\pi}\right)^n\int_{S^{n-1}}(e_0+is)\left(\langle {\bs x},s\rangle-x_0 s\right)^{-n}\/ds.
\end{equation}
If $x_0 < 0$, then
\begin{equation}\label{eqn:PW2}
E(x) = 
(-1)^{n+1}\frac{(n-1)!}{ 2}\left(\frac{i}{
2\pi}\right)^n\int_{S^{n-1}}(e_0+is)\left(\langle {\bs x},s\rangle-x_0 s\right)^{-n}\/ds.
\end{equation}
\end{thm}

Now suppose that $\bs A$ is a hyperbolic  $n$-tuple  of bounded linear operators on a Banach space $X$. Then
$\langle {\bs A},s\rangle-x_0 s I$ is an element of the space $\cL_{(n)}(X_{(n)})$ of module homomorphisms for each $s\in S^{n-1}$.
When $x_0\neq 0$ and $a\in\R$, the inverse of $(aI-\langle {\bs A},s\rangle)e_0-x_0 s I$ in the Clifford module  $\cL_{(n)}(X_{(n)})$ is given by
$$((aI-\langle {\bs A},s\rangle)e_0-x_0 s I)^{-1} = \left((aI-\langle {\bs A},s\rangle)+x_0 s I\right)((aI-\langle {\bs A},s\rangle)^2+x_0^2I)^{-1}.$$
Because $\s(\langle \xi, \bs A\rangle)\subset \R$, the Spectral Mapping Theorem ensures that
the bounded linear operator $(aI-\langle {\bs A},s\rangle)^2+x_0^2I$ is invertible and
$$\s((aI-\langle {\bs A},s\rangle)^2+x_0^2I) = \varphi_a(\s(\langle s,{\bs A}\rangle))$$
for the function $\varphi_a:t\longmapsto (a-t)^2+x_0^2$, $t\in\R$. Moreover,
$$((aI-\langle {\bs A},s\rangle)e_0-x_0 s I)^{-n} = \left((aI-\langle {\bs A},s\rangle)e_0+x_0 s I\right)^n((aI-\langle {\bs A},s\rangle)^2+x_0^2I)^{-n}$$
 in $\cL_{(n)}(X_{(n)})$. To define $G_x(\bs A)$, the promised substitution $\bs y \longrightarrow\bs A$
in the Cauchy kernel $G_x(\bs y )=E(x-\bs y )$ is now made by setting
\begin{equation}\label{eqn:CK1}
G_x(\bs A) = 
\frac{(n-1)!}{ 2}\left(\frac{i}{
2\pi}\right)^n\int_{S^{n-1}}(e_0+is)\left((\langle {\bs x},s\rangle-x_0 s)I-\langle {\bs A},s\rangle\right)^{-n}\/ds
\end{equation}
for $x= x_0e_0+\bs x$ with $\bs x\in \R^n$ and $x_0 > 0$, and
\begin{equation}\label{eqn:CK2}
G_x(\bs A) = (-1)^{n+1}
\frac{(n-1)!}{ 2}\left(\frac{i}{
2\pi}\right)^n\int_{S^{n-1}}(e_0+is)\left((\langle {\bs x},s\rangle-x_0 s)I-\langle {\bs A},s\rangle\right)^{-n}\/ds
\end{equation}
for $x_0 < 0$.
The expression $\left((\langle {\bs x},s\rangle-x_0 s)I-\langle {\bs A},s\rangle\right)^{-n}$
is left and right monogenic in $\cL_{(n)}(X_{(n)})$ for the variable $x= x_0e_0+\bs x$ with $x_0\ne 0$, so
differentiating under the integral sign shows that $x\longmapsto G_x(\bs A)$ is itself two sided monogenic for $x_0\ne 0$. 

Note that for $n$ even, symmetry in the integral gives
\begin{align*}
&\int_{S^{n-1}}(e_0+is)\left((\langle {\bs x},s\rangle-x_0 s)I-\langle {\bs A},s\rangle\right)^{-n}\/ds
= e_0\int_{S^{n-1}}\left((\langle {\bs x},s\rangle-x_0 s)I-\langle {\bs A},s\rangle\right)^{-n}\/ds\\
&\qquad=
e_0\int_{S^{n-1}}\left((\langle {\bs x},s\rangle+x_0 s)I-\langle {\bs A},s\rangle\right)^{n}\left((\langle {\bs x},s\rangle I-\langle {\bs A},s\rangle)^2+ x_0^2I\right)^{-n}\/ds
\end{align*}
and for $n$ odd
\begin{align*}
&\int_{S^{n-1}}(e_0+is)\left((\langle {\bs x},s\rangle-x_0 s)I-\langle {\bs A},s\rangle\right)^{-n}\/ds
= i\int_{S^{n-1}}s\left((\langle {\bs x},s\rangle-x_0 s)I-\langle {\bs A},s\rangle\right)^{-n}\/ds\\
&\qquad=
i\int_{S^{n-1}}\left((\langle {\bs x},s\rangle+x_0 s)I-\langle {\bs A},s\rangle\right)^{n}\left((\langle {\bs x},s\rangle I-\langle {\bs A},s\rangle)^2+ x_0^2I\right)^{-n}\/s\,ds.
\end{align*}
\smallskip

\subsection{The McIntosh functional calculus} For a general matrix or operator $A$, the Riesz-Dunford formula
\begin{equation}\label{eqn:RD}
f(A) = \frac{1}{2\pi i}\int_{C}(\lambda I-A)^{-1}f(\lambda)\,d\lambda
\end{equation}
is valid for all functions $f$ holomorphic in a neighbourhood of the spectrum $\s(A)$ in the complex plane.
The simple closed contour $C$ surrounds $\s(A)$ and is contained in the domain of $f$.
The higher-dimensional analogue of the Riesz-Dunford formula (\ref{eqn:RD}) is then
\begin{equation}\label{eqn:fc} 
f(\bs A) = \int_{\partial \Omega}G_x(\bs A)\bs n(x)f(x)\,d\mu(x)
\end{equation} for the hyperbolic $n$-tuple $\bs A =
(A_1,\dots,A_n)$ of bounded linear operators on
a Banach space $X$, simply by  substituting the $n$-tuple $\bs A$ for the vector $x\in\{0\}\times \R^n$
in the Cauchy integral formula (\ref{eqn:CIF}). The price
paid is that Clifford regular functions have values in the Clifford algebra
$\C_{(n)}$, which is noncommutative for $n=2,3,\dots$ and the correspondence between Clifford regular functions and their holomorphic counterparts needs to be investigated. The algebra $\R_{(1)}$
is isomorphic to $\C$ and formulas (\ref{eqn:fc}) and  (\ref{eqn:RD}) coincide
in the case $n=1$.

Let $\bs A$ is a hyperbolic  $n$-tuple  of bounded linear operators on a Banach space $X$ so that $G_\cdot(\bs A)$ is defined by the plane wave decomposition.
If $x\longmapsto G_x(\bs A)$ has a continuous extension to a neighborhood in $\R^{n+1}$ of a point $\bs a\in \R^n$,
then $G_\cdot(\bs A)$ is actually monogenic in a neighborhood of $\bs a$ in $\R^{n+1}$ by the monogenic analogue of Painlev\'e's Theorem \cite[Theorem 10.6 p. 64]{BDS}. 

The \textit{joint monogenic spectrum} $\g(\bs A)$ of the $n$-tuple $\bs A$
is the subset of $\R^n$ for which $\{0\}\times\g(\bs A)$ is the set of singularities of the
Cauchy kernel $G_\cdot(\bs A)$. In the case that the $n$-tuple $\bs A$ satisfies the exponential growth estimates (\ref{eqn:exp}), the joint spectrum $\g(\bs A)$ coincides with supp$(\cW_{\bs A})$ \cite[Theorem 4.8]{J}. In the case that 
the $n$-tuple $\bs A$ consists of bounded selfadjoint operators, $\g(\bs A)$ equals the Gelfand spectrum of a certain commutative Banach algebra (operants) associated with $\bs A$ \cite{Nel}. 
The monogenic analogue of Liouville's Theorem ensures that the {joint spectrum} $\g(\bs A)$
is nonempty and compact \cite[Theorem 4.16]{J}. 

The \textit{joint spectral radius} $r(\bs A) =\sup\{|x|:x\in\g(\bA)\,\}$ of the hyperbolic $n$-tuple $\bA$ is the
radius of the smallest ball centred at zero containing $\g(\bA)$.

The joint spectrum $\g(\bs A)$ is the analogue of the spectrum $\s(A)$ of a single operator $A$
in the sense that it is the set of singularities of the Cauchy kernel $\lambda\longmapsto (\lambda I -A)^{-1}$
in the Riesz-Dunford formula (\ref{eqn:RD}), that is, the set of $\lambda \in \C$ for which 
$\lambda I -A$ is not invertible --- the first is the analytic viewpoint and the second is the algebraic viewpoint.

Given a left monogenic function $f$ defined in a neighbourood $U$ of the joint spectrum $\g(\bs A)$,
the element $f(\bs A)$ of $\cL_{(n)}(X_{(n)})$ is defined by formula $(\ref{eqn:fc})$ independently
of the oriented $n$-manifold $\pd\Omega$ such that $\g(\bs A)\subset \Omega$ and $\overline \Omega\subset U$.

A real analytic function $f:V\to\C$ defined in a neighbourhood $V$ of $\g(\bs A)$ in $\R^n$ has a unique two-side monogenic extension $\tilde f$ (the Cauchy-Kowaleski extension) to a neighborhood $U$ of $\{0\}\times\g(\bs A)$ in $\R^{n+1}$. The extension is provided by an expansion in a series of monogenic polynomials \cite{BDS}. Then 
the definition $f(\bs A):=\tilde f(\bs A)$ makes sense and does not depend on the domain $U$ of $\tilde f$
containing $\{0\}\times\g(\bs A)$. 

It is important to know that $f(\bs A)\in \cL(X)$ (where $\cL(X)\equiv \cL(X)e_0$) and what 
the bounded linear operator $p(\bs A)\in\cL(X)$ is in the case that $p$ is a polynomial in $n$ real variables.
The following results are taken from \cite[\S 4.3]{J}.

\begin{thm}\label{thm:3.3.1} Let $\bs A$ be a hyperbolic $n$-tuple of bounded operators acting on a
Banach space $X$. 

\begin{enumerate}[label ={\rm (\roman*)}]
\item  Let $p:\C\to\C$ be a polynomial and $\zeta\in\C^n$. Set $f(z) =
p(\langle z,\zeta\rangle )$, for all $z\in\C^n$. Then  $f(\bs A) = p(\langle \bs A,\zeta\rangle ).$
\item  Suppose that $k_1,\dots,k_n = 0,1,2,\dots, k = k_1+\cdots+k_n$ and
$f(x) = x_1^{k_1}\cdots x_n^{k_n}$ for all $x = (x_1,\dots,x_n)\in\R^n$. Then
$$f(\bs A) = \frac{k_1!\cdots k_n!}{ k!}\sum_\pi A_{\pi(1)}\cdots A_{\pi(k)},$$ where
the sum is taken over every map $\pi$ of the set $\{1,\dots,k\}$ into
$\{1,\dots,n\}$ which assumes the value $j$ exactly $k_j$ times, for each $j
=1,\dots,n$.
\item  Let $\Omega$ be an open set in $\R^{n+1}$ containing $\g(\bs A)$ with a smooth
boundary $\pd\Omega$. Then for all $\omega\notin\overline{\Omega}$, 
$G_\omega(\bs A) = \int_{\pd\Omega} G_\zeta(\bs A)\bs n(\zeta)G_\omega(\zeta)\,d\mu(\zeta).$
\item  Suppose that $U$ is an open neighbourhood of $\g(\bs A)$ in $\R^n$ and
$f:U\to\C$ is an analytic function. Then $f(\bs A)\in\cL(X)$. 
\end{enumerate}
\end{thm}

For a commuting $n$-tuple $\bs A$ of bounded operators with real spectra, the McIntosh functional calculus
$f\longmapsto f(\bs A)$ given by formula $(\ref{eqn:fc})$ coincides with Taylor's functional calculus
$\tilde f\longmapsto \tilde f(\bs A)$ for the holomorphic counterpart $\tilde f:\tilde U\to \C$ of the 
monogenic function $f:U\to \R^{n+1}$, that is, $\tilde U$ is an open subset of $\C^n$ containing 
$\g(\bs A)$ and $U$ is an open subset of $\R^{n+1}$ containing 
$\{0\}\times \g(\bs A)$ such that $\tilde f(\bs x) = f(\bs x)$ for every element $\bs x$
of the open subset $(\tilde U\cap\R^n)\cap \pi_{\R^n}(U\cap(\{0\}\times\R^n)$ of $\R^n$.
Here $\pi_{\R^n}$ is the projection $\pi_{\R^n}(x_0,x_1,\dots,x_n)= (x_1,\dots,x_n)$ for $x_1,\dots,x_n\in\R$.

\begin{thm} Let $\bs A$ be a commuting $n$-tuple  of bounded operators
acting on a Banach space $X$ such that $\s(A_j) \subset \R$ for all $j=1,\dots,n$. 

Then
$\g(\bs A)$ is the complement in $\R^n$ of the set of all
$\lambda\in\R^n$ for which the operator $\sum_{j=1}^n(\lambda_jI-A_j)^2$ is invertible in
$\cL(X)$. 

Moreover,
$\g(\bs A)$ is the Taylor spectrum of $\bs A$. If the complex valued function $f$ is
real analytic in a neighbourhood of
$\g(\bs A)$ in $\R^n$, then the operator $f(\bs A)\in \cL(X)$ coincides with the operator
obtained from Taylor's functional calculus {\rm \cite{Vascu}}.
\end{thm}

In the noncommuting case, there is no homomorphism properties for the McIntosh functional calculus, but it
does enjoy symmetry properties similar to those of the Weyl calculus $\cW_{\bs A}$ when it exists, that is,
when the exponential estimate (\ref{eqn:exp}) obtains.

The following general properties of the Weyl functional calculus \cite[Theorem 2.9]{A1},
suitably interpreted, are also enjoyed by the McIntosh functional calculus.

\begin{thm}\label{thm:monp} Let $\bs A$ be an hyperbolic $n$-tuple of bounded operators acting on a
Banach space $X$.

\begin{enumerate}[label = {\rm (\roman*)}]
\item  {\bf Affine covariance:} if $L:\R^n\to\R^m$ is an affine map, then
$\g(L\vec A)
\subseteq L\g(\vec A)$ and for any function $f$ analytic in a neighbourhood in $\R^m$
of $L\g(\vec A)$, the equality $f(L\vec A) = (f\circ L)(\vec A)$ holds.

\item {\bf Consistency with the one-dimensional calculus:} if $g:\R\to\C$ is
analytic in a neighbourhood of the projection $\pi_1\g(\vec A)$ of $\g(\vec A)$ onto the
first ordinate, and $f=g\circ\pi_1$, then $f(\vec A) = g(A_1)$. We also have
consistency with the $k$-dimensional calculus, $1 < k < n$.

\item  {\bf Continuity:} The mapping
$(T,f)\longmapsto f(T)$ is continuous for $T = \sum_{j=1}^nT_je_j$ from 
$\cL_{(n)}(X_{(n)})\times M(\R^{n+1},\C_{(n+1)})$ to
$\cL_{(n)}(X_{(n)})$ and from $\cL(X)\times H_M(\R^n)$ to $\cL(X)$.

\item  {\bf Covariance of the Range:} If $T$ is an invertible continuous
linear map on $X$ and
$TAT^{-1}$ denotes the
$n$-tuple with entries $TA_jT^{-1}$ for $j=1,\dots n$, then $\g(TAT^{-1}) = \g(\vec A)$
and 
$f(TAT^{-1}) = Tf(\vec A)T^{-1}$ for all functions $f$ analytic in a neighbourhood of
$\g(\vec A)$ in $\R^n$.
\end{enumerate}
\end{thm}

It is time to make the first sentence of the Introduction precise 
by specifying the  term {\it hyperbolic} for a system $\bA$ of unbounded operators in a Banach space $X$.
The object is to ensure that the Cauchy kernel $G_{x_0e_0+ \bx }(\vec A)$ is sensibly defined
for all $x_0\neq0$ and $\bx\in\R^n$. 

The set of $s\in S^{n-1}$ with nonzero coordinates $s_j$ for every $j=1,\dots,n$ 
is denoted by $S_0^{n-1}.$ Then $S_0^{n-1}$ is a dense open subset of 
$S^{n-1}$ with full surface measure. 

%
The notation $\cD(f)$ is used for the domain of a function $f$.

\begin{dfn}\label{dfn:hyp} 
An $n$-tuple $\bA =(A_1,\dots,A_n)$ of densely defined operators in a Banach space $X$ is called {\it hyperbolic} provided that 
\begin{enumerate}
\item $\cap_{j=1}^n\cD(A_j)$ is dense in $X$,
\item $\av As$ is closable on $\cap_{j=1}^n\cD(A_j)$ for every $s\in S_0^{n-1}$ and
\item $\s\big(\,\ol{\av As}\,\big)\subset\R$ for every $s\in S_0^{n-1}$.
\end{enumerate}
\end{dfn}
The operator
$\left((\langle {\bs x},s\rangle I-\langle {\bs A},s\rangle)^2+ x_0^2I\right)^{-n}$
makes sense for $s\in S_0^{n-1}$ because we can interpret it as the bounded linear operator
$$\left((\langle {\bs x},s\rangle+ i|x_0|) I-\ol{\langle {\bs A},s\rangle}\right)^{-n}\left((\langle {\bs x},s\rangle- i|x_0|) I-
\ol{\langle {\bs A},s\rangle}\right)^{-n}$$
and also
$$\left((\langle {\bs x},s\rangle+x_0 s)I-\langle {\bs A},s\rangle\right)^{n} \left((\langle {\bs x},s\rangle+ i|x_0|) I-\ol{\langle {\bs A},s\rangle}\right)^{-n}$$
is a bounded linear operator for every $x_0\neq 0$ and $\bx\in\R^n$. 

Because $S_0^{n-1}$ is a set of full measure $x\mapsto G_x(\bA)$
is defined off $\{0\}\times \R^n$ by formulas (\ref{eqn:CK1}) and (\ref{eqn:CK2}). Differentiation
under the integral sign ensures that $x\mapsto G_x(\bA)$ is two-sided monogenic in $\cL_{(n)}(X_{(n)})$
so that the joint spectrum $\g(\bA)$ is a closed and nonempty subset of $\R^n$. The same argument works if 
$S_0^{n-1}$ is replaced by a set of full Hausdorff measure.
If $X$ is a Hilbert space and elements of $\bA$ are selfadjoint, then (b) implies (c) in the Definition \ref{dfn:hyp}.

The next step is to verify that (\ref{eqn:fc}) produces the right result for elementary functions. For bounded operators
the case of polynomials is treated in Theorem \ref{thm:3.3.1}. For unbounded operators we should check
functions like $x\mapsto p(\as x\xi)(\lambda - \as x\xi)^{-k}  $, $x\in\R^n$, for $\lambda\in\C$, $\Im\lambda\ne 0$ with
$p$ a polynomial of degree less than or equal to $k=1,2,\dots$\,.
As we are concerned here only with matrices we won't go any further into the case of unbounded operators.

\section{Systems of Matrices} 
Let $\vec A = (A_1,\dots,A_n)$ be a hyperbolic $n$-tuple of $N\times N$ matrices. If $A_1,\dots,A_n$ are hermitian
then the  exponential bound (\ref{eqn:exp}) follows from the Lie-Kato-Trotter product formula with
$C=1$, $s=0$ and $r=\|\vec A\|$ \cite[Theorem 1]{Taylor}.

In general, the  exponential bound (\ref{eqn:exp}) follows from the properties of hyperbolic polynomials
considered below.  It is instructive to see this directly. The {\it characteristic polynomial} of
a square matrix $B$ is $p_B(z)=\det(B-zI)$, $z\in\C$.

\begin{prp}\label{prp:exbd} 
Let $\bA$ be a hyperbolic $n$-tuple of $N\times N$ matrices. Then for each $r > \|\vec A\|$, there exists $C>0$ 
such that the exponential bound {\rm (\ref{eqn:exp})} holds with $s=N-1$. 
\end{prp}

\begin{proof} Following  \cite[V \S 2]{John}, to see directly that $\bA$ satisfies the bound (\ref{eqn:exp})
 we can take the Fourier transform of (\ref{eq:Weyl}), estimate the factor associated with a differential
 operator of order $s = N-1$ and look at the function
$$Z(\xi,t) = \frac1{2\pi}\int_{C(\xi)}\frac{e^{iz t}}{p_{\av{A}{\xi}}(-z)}\,dz,$$
for a suitable contour $C(\xi)$ with $\xi\in\R^n$.
Let $D_r$ be the closed disk of radius $r\ge 0$ centred at zero in $\C$ ($D_0=\{0\}$).
If we know that $\s(\av A\xi) \subset D_{r(\xi)}$ with $r(\xi)\ge 0$ for $\xi\in\R^n$ and 
$$r =\sup_{|\xi|=1}r(\xi) < \infty,$$ 
then we can take $C(\xi)$
to be that part of the circle of radius $|\xi|(r+1)$ centred at zero in  the half-plane $\{z:\Im z >-1\}$ joined with a segment of the line $\{z:\Im z = -1\}$, oriented clockwise. Then $\text{dist}(C(\xi),\s(\av A\xi))\ge 1$ for each $\xi\in\R^n $, so
$$|Z(\xi,t)| \le  \frac{e^{t}}{2\pi }\int_{C{(\xi)}}\frac{|dz|}{|p_{\av A{\xi}}(-z)|}\le e^t|\xi|(r+1),\qquad  \xi\in\R^n.$$
With $t=1$, an appeal to the Payley-Weiner Theorem shows that $\bs A$ is of type $(s,r)$,
the Weyl fuctional calculus $\cW_{{\bs A}}$ exists and $\g({\bs A}) =\supp(\cW_{{\bs A}})$. It suffices to take
$$r(\xi) = \|\av A\xi\|\le \|\bs A\|\,|\xi|.$$

The number $r>0$ in the bound (\ref{eqn:exp}) may be taken to be any number greater than the spectral radius
$r({\bs A}) = \sup\{|x|:x\in\g({\bs A})\}$ and $s = N-1$.
\end{proof}

The speed of propagation is equal to the joint spectral radius $r(\vec A)$ \cite[p. 131]{John}. The inclusion 
$$\s(\av A\z) \subset r(\z)D_{|\z|},\quad \z\in\C^n$$
with the $r(\z)$ given by equation 
$$r(\z) = \sup\{|x|:x\in\g((\langle \bs A,\xi\rangle,\langle \bs A,\eta\rangle)\},\quad \z=\xi+i\eta,\ \xi,\eta\in\R^n,$$ 
is a general fact about 
bounded hyperbolic operators on a Banach space \cite[Theorem 5.7]{J}. The relevant result for hyperbolic
polynomials is \cite[Theorem 12.5.1]{H2}.

\subsection{The Fundamental Solution}
Let 
\begin{eqnarray*}
P^{\vec A}(\z_{0},\z_1,\cdots,\z_n) &=& \det(\z_{0} I+\z_1A_1+\cdots+\z_nA_n)\\
&=& p_{\langle \vec A,\bs\z\rangle }(-\z_{0}),
\end{eqnarray*}\glossary{$P^{\vec A}$}
for all $\z\in\C^{n+1}$ with the representation $\z = \z_{0}e_0+\bs\z$,
$\bs\z\in\C^n$.

Let $\RP n$ be real $n$-dimensional projective space. Then
\begin{equation}\label{eqn:4.4}
\Xi(\vec A) =\ \SetOf{(\xi_0:\xi_1:\cdots:\xi_n)\in \RP n}
                 {P^{\vec A}(\xi_0,\xi_1,\cdots,\xi_n)=0}
\end{equation}
\glossary{$\Xi(\vec A)$}is an algebraic hypersurface. Identifying elements of $\RP n$
with lines $[b]$ in $\R^{n+1}$, $b\in \R^{n+1}\setminus\{0\}$, let $ \G(P^\bA)$\glossary{$\G(P^\bA)$} denote the
open connected component of $\R^{n+1}\setminus \Xi(\vec A)$
containing $e_0$.
The convex cone $ \G(P^\bA)$ in $\R^{n+1}$ is called the 
{\it hyperbolicity cone\index{hyperbolicity cone}\/} of $\vec A$.
The trace of the dual cone $K(P^{\vec A})$ of $\G(\vec A)$
on the set $x_0 = 1$, referred to as the {\it propagation set\index{propagation!set|(}\/} of $\vec A$, 
is the given by
\begin{equation}\label{eqn:4.61}
K(\vec A) =
\ \SetOf{\bx\in\R^{n}}{\langle e_0+\bx,\xi\rangle \ge 0,\ \forall \xi\in \G(\vec A)\  }. 
\end{equation}\glossary{$K(\vec A)$}
In the case that $n=2$ and $\vec A =(A_1,A_2)$ is a pair of hermitian matrices,
then the result  of Kippenhahn mentioned in
Section 2 ensures that the set $K(\vec A)$
can be identified with the numerical range of the matrix $A = A_1+iA_2$. Writing $\cF\phi=\hat\phi$, 
$\phi\in \cS(\R^{n+1})$ for the Fourier transform, its inverse map is
$$\big(\cF^{-1}\psi\big)(x) = \frac1{(2\pi)^{n+1}}\int_{\R^{n+1}}e^{i\as x\xi}\psi(\xi)\,d\xi,\quad \psi\in \cS(\R^{n+1}),\ x\in\R^{n+1}$$
and by a common abuse of notation the dual linear maps on $\cS'(\R^{n+1})$  are denoted by the same symbols.

The distribution
$$\cF^{-1}\left(\frac{1}{P^\bA(\xi-i0)}\right)(te_0+\bx)$$
is the fundamental solution of $P^\bA(\tau e_0 +\vec D) =\d_0$ in the sense that
$$\cF^{-1}\left(\frac{1}{P^\bA(\xi-i\e e_0)}\right)(te_0+\bx)$$
is defined as a distribution for all $\e > 0$ and is independent of $\e$ \cite[(12.5.3)]{H2}.
By the Paley-Wiener-Schwartz theorem 
\begin{equation}\label{eq:supp}
\co\left(\supp\left(\cF^{-1}\left(\frac{1}{P^\bA(\xi-i0)}\right)\right)\right) = K(P^\bA)
\end{equation}
\cite[Theorem 12.5.1]{H2},
so that 
$$\cF^{-1}\left(\frac{1}{P^\bA(\xi-i\e e_0)}\right)(te_0+\bx) = 0,\quad t < 0.$$
Furthermore
$$\frac{1}{P^\bA(\xi-i0)} =\lim_{\e\to0+}\frac{1}{P^\bA(\xi-i\e e_0)}$$
as the limit of distributions in $\cS'(\R^n)$ \cite[Theorem 4.1]{ABG1}. Similarly,
$\cF^{-1}\left(\frac{1}{P^\bA(\xi+i0)}\right)$  is supported in $-K(P^\bA)$.
The equality (\ref{eq:supp}) also establishes Proposition \ref{prp:exbd}.

\subsection{The Numerical Range Distribution}
Let $n_{\vec A}$ be the joint numerical range map (\ref{eq:jnr}) for the $n$-tuple $\bA$ of $(N\times N)$ 
hermitian matrices and $\mu$ the unitarily invariant
probability on $S(\C^N)$. Then the Borel probability
measure $\nu_{\vec A} = \mu\circ n_{\vec A}^{-1}$ is supported by the joint numerical range 
$N_\vec A = n_\bA(S(\C^N))$ of $\bA$. The probability measure $\nu_\bA$ on $(\R^n,\cB(\R^n))$ is called the
{\it numerical range distribution} of the $n$-tuple $\bA$. For general hyperbolic matrices $\bA$, the measure $\nu_\bA$
lives on $(\C^n,\cB(\C^n))$.

Let $P^\bA(\xi)=\det(\xi_0 I+\av\bA{\bs\xi})$, $\xi\in \R^{n+1}$. Nelson's representation (\ref{eq:nel})
suggests that the numerical range distribution $\nu_\bA$ and the fundamental solution of 
$F_{P^\bA}$ of $P^\bA(\tau,\vec D) = \d_0$ must be related. The relationship is curiously the basis of the
generalisation of the Cauchy integral formula considered in the paper \cite{J2} which intentionally does not mention the fundamental solution $F_{P^\bA}$.

The {\it adjugate} of an invertible square matrix
$B$ is the matrix adj$B=(\det B)B^{-1}$, that is, the transpose of the matrix of signed minor determinants  of $B$.
The same rules apply if $B$ is a matrix of differential operators acting on $\cS'(\R^m)$, $m=1,2,\dots$\,.


\begin{thm}\label{thm:5.2} Let  $\bA$ be an $n$-tuple of $(N\times N)$ 
hermitian matrices. The equality
 $$ t^{N-1}\nu_{t\vec A} =  (-i)^{N}(N-1)!\cF^{-1}\left(\frac{1}{P^\bA(\xi-i0)}\right)(te_0+\var),\quad t > 0.$$
 holds in the sense of distributions on $(0,\infty)\times\R^n$ and
 $$ \nu_{\vec A} =  (-i)^{N}(N-1)!\cF^{-1}\left(\frac{1}{P^\bA(\xi-i0)}\right)(e_0+\var)$$
 as distributions on $\R^n$. Furthermore
 \begin{eqnarray}\label{eq:WABG}
\cW_{t\bA} &=& i \text{\rm adj}(\pd_t I + \av A{\nabla})\cF^{-1}\left(\frac{1}{P^{\bA}(\xi-i0)}\right)(te_0+\bx)\\
&=&i^N\text{\rm adj}(\tau I + \av A{\bs D})\cF^{-1}\left(\frac{1}{P^{\bA}(\xi-i0)}\right)(te_0+\bx) ,\quad t>0.\nonumber
\end{eqnarray}
\end{thm}
\begin{proof} According to \cite[equation (5.7)]{J} we have
$$\frac{1}{2\pi i}\int_{C(\xi)}\frac{e^{iz}}{p_{\langle \vec A,\xi\rangle}(z)}\,dz = -\frac{(-i)^{N-1}(2\pi)^n}{(N-1)!} \check\nu_{\vec A}(\xi)$$
for $\xi\in\R^n$ and a simple closed contour $C(\xi)$ about the real eigenvalues $\s(\as\bA\xi)$.
 
This equation arises in a proof of Nelson's representation for the Weyl calculus \cite[Proposition 5.4]{J}  (a forward slash
is obviously missing in equations (5.4) and (5.5) there) but was not mentioned by Nelson himself.
For $N\ge 2$, the characteristic polymomial has only real roots so that the left hand side is
$$\frac{1}{2\pi i}\int_{(\R-i\eta) + (-\R+i\eta)}\frac{e^{iz}}{p_{\langle \vec A,\xi\rangle}(z)}\,dz$$
for any $\eta > 0$, with the understanding that $\R$ is oriented from $-\infty$ to $\infty$.

Taking a Fourier transform and substituting $z=-\z$ yields
$$\frac{1}{2\pi i}\int_{(\R-i\eta) + (-\R+i\eta)}\int_{\R^n}\frac{e^{-i\z}\hat\phi(\xi)}{p_{\langle \vec A,\xi\rangle}(-\z)}\,d\xi\,d\z = \frac{(-i)^{N-1}(2\pi)^n}{(N-1)!} \int_{\R^n}\hat\phi(\xi)\check\nu_{\vec A}(\xi)\,d\xi$$
$$\frac{1}{(2\pi)^{n+1} }\int_{(\R-i\eta)+(-\R+i\eta)}\int_{\R^n}\frac{e^{-itz}\hat\phi(\xi)}{P^\bA(ze_0+\xi)}\,d\xi\,dz = -\frac{(-i)^{N}t^{N-1}}{(N-1)!} \int_{\R^n}\hat\phi(\xi)\check\nu_{t\vec A}(\xi)\,d\xi,\quad t>0,$$
and all $\phi\in\cS(\R^n)$.
\begin{eqnarray*}
t^{N-1}\nu_{t\vec A} &=& -i^{N} (N-1)!\cF^{-1}\left(\frac{1}{P^\bA (\xi-i0)}-\frac{1}{P^\bA (\xi+i0)}\right)(-te_0-\bx)\\
&=& {i^{N}(N-1)!}\cF^{-1}\left(\frac{1}{P^\bA (\xi+i0)}\right)(-te_0-\bx)\\
&=& {(-i)^{N}(N-1)!}\cF^{-1}\left(\frac{1}{P^\bA (\xi-i0)}\right)(te_0+\bx).
\end{eqnarray*}
as distributions. 
The distribution $\cF^{-1}\left(\frac{1}{P^\bA(\xi-i0)}\right)$ is supported in $K(P^\bA)$ while $-te_0-\bx\in -K(P^\bA)$ and $\cF^{-1}\left(\frac{1}{P^\bA(\xi+i0)}\right)$  is supported in $-K(P^\bA)$ as mentioned above.
Convolution with $\d_1\otimes\d_0$ gives the second equality. Using similar reasoning
\begin{eqnarray*}
e^{it\av A\xi} &=& \frac1{2\pi i}\int_{C(\xi)}e^{it\lambda}(\lambda - \av A\xi)^{-1}\,d\lambda\\
&=& -\frac1{2\pi i}\int_{(\R-i\eta) + (-\R+i\eta)} e^{it\lambda}\frac{\text{adj}(-\lambda I + \av A\xi)}{P^\bA(-\lambda e_0 + \xi)}\,d\lambda
\end{eqnarray*}
from which formula (\ref{eq:WABG}) follows by taking the Fourier transform.
\end{proof}
\begin{xmp}
If $\bs\s$ are the Pauli matrices, then $P^{\bs \s}(\xi) = \xi_0^2-|\vec \xi|^2$ for $\xi\in\R^{n+1}$.
According to \cite[p. 204]{GS},  the solution of the wave equation in $\R^4$  is
$$\delta(r-t)/(4\pi t)=-\cF^{-1}\left(\frac{1}{P^{\bs\s}(\xi-i0)}\right)(te_0+\bx).$$
The negative sign in the formula is because the fundamental solution of the wave equation
satisfies $p(\tau,\bs D)E = \delta_0$ for the hyperbolic polynomial $p(\xi)=-\xi_0^2+|\bs \xi|^2.$ 
The convention is required when considering nonhomogeneous hyperbolic polynomials 
$p$---these are not
needed in the present context of hyperbolic systems $\bA$.

Computing $\nu_{t\bs\s}$ requires the evaluation of an integral in polar coordinates 
for $S(\R^4)$ but we can simply read it off formula   (\ref{eq:Pauli}) and Nelson's formula \cite[Theorem 5.1]{J} given by
\begin{eqnarray*}
\cW_{t\vec{\s}} &&=\sum_{k=0}^{1}\sum_{j=0}^{1-k}\sum_{m=0}^j(-1)^{k+m}
\left(\begin{array}{c}j\\m\end{array}\right)\frac{1}{ (1-j+m)!}\times\cr
&&\hskip1cm\langle t\vec \s,\nabla\rangle ^k\phi_{1-j-k}(\langle t\vec \s,\nabla\rangle )(\nabla\cdot
id)^m\nu_{t\bs \s}\\
 &&=\sum_{j=0}^{1}\sum_{m=0}^j(-1)^{m}
\left(\begin{array}{c}j\\m\end{array}\right)\frac{1}{ (1-j+m)!}\times\cr
&&\hskip1cm \phi_{1-j}(\langle t\vec \s,\nabla\rangle )(\nabla\cdot
id)^m\nu_{t\vec \s} - \langle t\vec \s,\nabla\rangle\nu_{t\vec \s}\\
&&=\phi_{1}(\langle t\vec \s,\nabla\rangle )\nu_{t\vec \s}+\nu_{t\vec \s}-(\nabla\cdot
id)\nu_{t\vec \s} - \langle t\vec \s,\nabla\rangle\nu_{t\vec \s}\\
&&=\nu_{t\vec \s}-(\nabla\cdot
id)\nu_{t\vec \s} - \langle t\vec \s,\nabla\rangle\nu_{t\vec \s}.
\end{eqnarray*}
to verify
$$t\nu_{t\bs\s} = t\delta(r-t)/(4\pi t^2)=-\cF^{-1}\left(\frac{1}{P^{\bs \s}(\xi-i0)}\right)(te_0+\bx).$$
\end{xmp}
Further identities follow from the Nelson formula 
\begin{eqnarray}
\cW_{t\vec{A}}
&=&\sum_{k=0}^{N-1}\sum_{j=0}^{N-k-1}\sum_{m=0}^j(-1)^{k+m}
\left(\begin{array}{c} j\\m\end{array}\right)\frac{1}{(N-1-j+m)!}\times\\\label{eq:nel}
&&\hskip3cm\langle t\vec A,\nabla\rangle ^k\phi_{N-j-k-1}(\langle t\vec A,\nabla\rangle )(\nabla\cdot
id)^m\nu_{t\vec A} \nonumber
\end{eqnarray}
by equating coefficients in powers of $t$ with $i \text{\rm adj}(\pd_t I + \av A{\nabla})\cF^{-1}\left(\frac{1}{P^{\bA}(\xi-i0)}\right)(te_0+\bx)$.

\section{Lacunas} 
The Herglotz-Petrovsky-Leray
formula for the fundamental solution $E$ of the hyperbolic differential operator $P^\bA(\tau,\bs D)$ 
for $P^\bA(\z) =\det(\z_0I+\langle \vec A,\vec \z\rangle)$, $\z\in\R^{n+1}$ is derived from
the distribution 
$$\cF^{-1}\left(\frac{1}{P^\bA(\xi-i0)}\right)$$
 by the explicit integration
of the radial variable.
In particular for the hermitian case of $N\ge n+1$, it turns out that ${d\nu_{\bA}}/{d\lambda}$ is a polynomial of degree 
$N-n-1$ in regions where the {\it Petrovsky cycle} vanishes. Such regions are called {\it lacunas}
because  they make up the connected components of $\co(N_\bA)^\circ\setminus\supp(\cW_\bA)$.

In the hermitian cases $n=2$, $N\ge 2$ and $n=3$, $N\ge 3$, the joint numerical range $N_\bA$ is convex and the set
$(N_\bA)^\circ\setminus\supp(\cW_\bA)$ consists of the gaps between the joint numerical range $N_\bA$ and the support $\supp(\cW_\bA)$ of the Weyl functional calculus. The trivial gaps are eliminated by
restricting the elements of $\bA$ to the subspaces $X$ of $\C^N$ on which $\bA$ has no further
nontrivial joint invariant subspaces of $X$. By spectral theory and induction $\C^N$ can be written as the orthogonal
sum of subspaces on which $\supp(\cW_\bA)$ has no further decomposition.

The linearised equations of magnetohydrodynamics correspond to the case $n=2$ and $N=7$
for hermitian matrices. A beautiful instrument drawing before the era of consumer computer graphics
appears in Figures 1a, 1b of the fundamental paper of J. Bazer and D. Yen \cite{BY1}.
The corresponding symmetric matrices have no joint eigenvalues or nontrivial joint invariant subspaces.

Let $n$ be an even integer and $\vec A =
(A_1,\dots,A_n)$ a hyperbolic $n$-tuple of $(N\times N)$ matrices. Modifications required for the case of odd $n$ 
are indicated later.
The purpose of this section
is to outline a general method using Clifford analysis for establishing that
a point $\bx\in\R^n$ belongs to the joint spectrum $\g(\vec A)$ or not and to determine if
$\supp(\cW_\bA) = K(\bA)$ when $\bA$ has no nontrivial joint invariant subspaces.
Because the propagation set $K(\bA)$ is convex by construction and 
$\supp(\cW_\bA) \subseteq K(\bA)$ \cite[Theorem 4.1]{ABG1}, the set $K(\bA)\setminus \supp(\cW_\bA) $
consists of genuine lacunas for the Weyl functional calculus.
%

Roughly speaking, the approach of Atiyah, Bott and
G\"arding \cite{ABG1} is interpreted in the present 
matrix setting,\cite[Section 5.3]{J}
and we see that the detailed explanation given in \cite{J}  for the fundamental case $n=2$
may be generalised by using the appropriate tools from
algebraic topology.
The presentation of this section is based on the summary
of the
Herglotz-Petrovsky-Leray formulas \cite{ABG1} given
by Y. Berest in \cite{Ber}. Another brief account is given in \cite[Section 12.6]{H2}.

A general element $x =
(x_0,x_1,\dots,x_n)$ of $\R^{n+1}$ will be written as
$x  = \bx + x_0e_0 $ with $\bs x = \sum_{j=1}^nx_je_j$. Because $n$ is
assumed to be an even integer,
$$\int\limits_{S^{n-1}}^{}s\left(\langle \bs xI -
\vec A,s\rangle -x_0s\right)^{-n}\/ds=0$$ and the plane wave
decomposition (\ref{eqn:CK1}), (\ref{eqn:CK2}) for the Cauchy kernel is
\begin{eqnarray} G_{x}(\vec A) &=& \cW_{\vec A}(G_x)\nonumber\\
&=&\frac{(n-1)!}{ 2}\left(\frac{i}{
2\pi}\right)^n\hbox{sgn}(x_0)\int\limits_{S^{n-1}}^{}\left(\langle \bs xI -
\vec A,s\rangle -x_0s\right)^{-n}\/ds,\label{eqn:4.57}
\end{eqnarray} 
for $x\in\R^{n+1}$ with $x_0\neq 0$. For ease of notation,
an element
$x I$ of $\cL_{(n)}(\C^N)$ for $x\in \C_{(n)}$ will often be written as $x$.
Because $x\longmapsto G_{x}(\vec A)$ is actually the Cauchy transform
of the Weyl calculus $\cW_{\vec A}$ off $\R^n$, we have
\begin{equation}\label{eqn:4.58}
\cW_{\vec A} = \lim_{\e\to 0+}G_{\bx + \e e_0}(\vec A) -G_{\bx - \e
e_0}(\vec A)
\end{equation} in the sense of distributions. Consequently, if the limit on
the right hand side of equation  (\ref{eqn:4.58}) exists uniformly for all
$\bs x$ in a neighbourhood of a point
$\bs a\in \R^n$ and  is zero there, then 
$\bs a$ lies outside the support of the matrix valued distribution
$\cW_{\vec A}$, that is, $\bs a\in \g(\vec A)^c$. We shall seek conditions
which guarantee that the limit
\begin{equation}\label{eqn:4.59}
\lim_{\e\to 0+}\int\limits_{S^{n-1}}^{}\left(\langle \bs xI -
\vec A,s\rangle -\e s\right)^{-n} +\left(\langle \bs xI -
\vec A,s\rangle +\e s\right)^{-n}\/ds
\end{equation} exists uniformly for all elements $\bs x$ of an
open subset of $\R^n$. In any case $\cW_{\vec A}$ is $\frac{(n-1)!}{ 2}\left(\frac{i}{
2\pi}\right)^n$ times the limit (\ref{eqn:4.59}) in the distributional sense.

For the case $n=2$ considered in \cite[Section 5.3]{J},  the integral
(\ref{eqn:4.59}) was calculated in an elementary manner by converting it
into a contour integral and actually computing the residues associated with 
the spectral representation of the hermitian matrix 
$\langle \vec A,s\rangle $ following the analysis of Bazer and Yen \cite{BY1}.

\subsection{Hyperbolic polynomials}
In this subsection properties and concepts of hyperbolic polynomials are stated
in the context of the determinantal polynomial 
$$P^\bA(\z)= \det(\z_0I+\av A{\bs\z}),\quad \z=\z_0+\bs\z\in\R^{n+1}$$ associated with a hyperbolic
$n$-tuple $\bA$ of matrices. Most phenomena are already exhibited in the classes of 
simultaneously triangularisable matrices
with real spectra, hermitian matrices and their direct sums.

A {\it localisation\index{localisation}} $P^{\vec A}_{\xi}$\glossary{$P^{\vec A}_{\xi}$}
 of $ P^{\vec A} $ at $ \xi\in\R^{n+1} $, 
is the lowest nonzero term of the polynomial
$$t \mapsto P^{\vec A}( \xi + t\zeta ) = t^{\mu_{\xi}} P^{\vec A}_{\xi}(\zeta) + {\mathcal 
O}(t^{\mu_{\xi}+1})\ , 
\quad  \mu_{\xi} = \deg\,P^{\vec A}_{\xi}\ .$$
Let $ \vec A $  and $ \xi \in \R^{n+1} $ be fixed. 
Consider
the localisation  $ P^{\vec A}_{\xi} $ of $ P^{\vec A} $ at $ \xi $. 
The {\it local hyperbolicity cone\index{local hyperbolicity cone}} 
and
the {\it local propagation set\index{local propagation set}} of $ P^{\vec A} $ at $ \xi $ 
are defined by setting, respectively,
\begin{equation*}
\Gamma_{\xi}({\vec A}) := \Gamma(P^{\vec A}_{\xi}),\qquad 
K_{\xi}({\vec A}) := 
\{\bx\in\R^n:[e_0+\bx]\in K(P^{\vec A}_{\xi})\ \} \ .
\end{equation*}\glossary{$\Gamma_{\xi}({\vec A})$}\glossary{$K_{\xi}({\vec A})$}
Here the polynomial $P^{\vec A}$ has been replaced by
$P^{\vec A}_\xi$ in the definitions (\ref{eqn:4.4})
and (\ref{eqn:4.61}). A similar notation is used for
the real lineality\index{lineality|)} $\Lambda(P^{\vec A}_{\xi})$
of the polynomial $P^{\vec A}_{\xi}$.

Clearly, $ \Gamma_{\xi}({\vec A}) \supseteq \Gamma(\vec A) $ and, 
hence, 
$ K_{\xi}({\vec A}) \subseteq K( \vec A ) $ for all $ \xi 
\in \R^{n+1} $.
More precisely, the mapping $ (\xi, {\vec A}) \mapsto \Gamma_{\xi}({\vec A}) $ 
(and $ (\xi, {\vec A}) 
\mapsto K_{\xi}({\vec A}) $) is {\it inner} (resp., {\it outer})
{\it continuous} in the sense that $ \Gamma_{\xi}({\vec A})\, \cap \, 
\Gamma_{\tilde\xi}(\tilde {\vec A}) $
(resp., $ K_{\xi}({\vec A})\, \cup  \, K_{\tilde\xi}(\tilde {\vec A}) 
$) is close to 
$ \Gamma_{\xi}({\vec A}) $ (resp., $ K_{\xi}({\vec A}) $) when
 $ (\tilde\xi, \tilde {\vec A}) $ is close to $ (\xi, {\vec A}) $ with
$ \xi, \tilde\xi \in \R^{n+1} $ and 
$ {\vec A}, \tilde{\vec A} $ hyperbolic.

\begin{xmp}\label{xmp:tan} Suppose that $\bA$ consists of two hermitian matrices. Then $\xi \in D_\R(\vec A)$ is a simple point
if and only if $P^{\vec A}_\xi(\z) = \langle b,\z\rangle$ where
$b_0\ne 0$ and $b$ is the tangent vector at $\xi$, that is 
$$[b] =\left [\frac{\pd P^{\vec A}}{\pd \xi_0}(\xi): \frac{\pd P^{\vec A}}{\pd \xi_1}(\xi):\frac{\pd P^{\vec A}}{\pd \xi_2}(\xi)\right] .$$
Then $K_\xi({\vec A})= \{\bx:[e_0+\bx]=[b]\}$ and 
$$C(\vec A) = \overline{\bigcup_{\mu_\xi(\vec A) = 1}K_\xi({\vec A})}.$$
\cite[Theorem 14.20]{ABG2},  shows that ss$(\cW_{\vec A}) = C(\vec A)$.

If $\xi \in D_\R(\vec A)$ and $\deg P^{\vec A}_\xi=2$ (the multiplicity of tangent vectors to $D_\R(\vec A)$ at $\xi$) and
$$P^{\vec A}_\xi(\z) =a(\z_0 +\langle b_1,\bs\z\rangle)(\z_0 +\langle b_2,\bs\z\rangle)$$
for some $a \in\R$, $b_1,b_2\in\R^2$ , then $b_1,b_2\in C(\bA)$ and
$$\G(P^{\vec A}_\xi) = \{\z_0 +\langle b_1,\bs\z\rangle >0,\ \z_0 +\langle b_2,\bs\z\rangle >0\ \},\ 
K_\xi({\vec A}) = \co\{b_1,b_2\}.$$
Hence $K_\xi({\vec A})$ is the line segment joining $b_1$ and $b_2$ lying on the ``double tangent" corresponding to
$\xi\in \R^3$.
\end{xmp}

At this stage, we need to take into account that 
the homogeneous polynomial $ P^{\vec A} $
may not depend on all variables in $\C^{n+1}$.
For example, one of the matrices $A_j$ could be the zero matrix.

The {\it real lineality\index{lineality|(}}\/ $ \Lambda( \vec A ) $ 
\glossary{$\Lambda( \vec A )$} of $ {\vec A}$, is the maximal linear subspace of $ \R^{n+1} $ 
such that the restriction of $ P^{\vec A} $ 
the quotient $ \R^{n+1}/ \Lambda( \vec A ) $ is again
a  polynomial.  Then $ \Lambda( \vec A ) $
coincides with the edge of the hyperbolicity cone\index{hyperbolicity cone} $ \Gamma(\vec A) $, 
so that
 $ \Gamma + \Lambda = \Gamma $, and $ K( \vec A ) $ spans 
the intersection of its orthogonal 
complement
$ \Lambda^{\perp}( \vec A ) $ in $ \R^{n+1} $ with
the plane $x_0=1$.

The system $\vec A$ is called {\it complete\index{complete set of matrices}}
if  $ {\vec A} $ has a trivial lineality. In this case, $ P^{\vec A}_{\xi}(\zeta) 
\equiv P^{\vec A}(\zeta) $ implies
$ \xi = 0 $, the cone $ \Gamma(\vec A) $ is proper (peaked) in the sense 
that $ \overline{\Gamma(\vec A)} $ 
does not contain any straight lines, and then $ K( \vec A ) $ has a 
non-empty interior
$ K^{\circ}( \vec A ) $ in $\R^n$.

The {\it wave front surface\index{wave front surface}} $ W({\vec A}) $\glossary{$W({\vec A})$} of the system
$ {\vec A} $ of matrices
is generated by the union of local propagation cones:
\begin{equation}\label{eqn:4.62}
W({\vec A}) := \bigcup\limits_{0 \not=\xi \in \R^{n+1}}^{}\, 
K_{\xi}({\vec A})\ .
\end{equation}
\begin{thm}[Joswig and Straub \cite{JS}] Let $\bA$ be two $(N\times N)$ hermitian matrices. The set of
critical points of the numerical range map $n_{\bA}:S(\C^N) \to \R^2$ is $n_\bA^{-1}(W({\bA}))$.
\end{thm}

By the implicit function theorem, the numerical range distribution $\nu_\bA$ has a real analytic
density with respect to Lebesgue measure outside the wave front set $W({\bA})$.
Although the distribution $\cW_\bA$ consists of a differential operator of order $(N-1)$ acting on
the numerical range distribution $\nu_\bA$ as may be seen from Nelson's representation (\ref{eq:nel}),
the equality ss$(\cW_{\vec A}) = C(\vec A)$ proved in \cite[Theorem 14.20]{ABG2},  shows that $\cW_{\vec A}$
does not see the double tangents mentioned in Example \ref{xmp:tan} in the hermitian case---a fact already apparent from the simplest of examples.

For an ordered set $a= (a_1,\dots,a_N)$ of real numbers, $\diag a$ denotes the $(N\times N)$ diagonal matrix with
entries $a_1,\dots,a_N$ down the diagonal.

\begin{xmp}\label{xmp:wf}
\begin{enumerate}
\item Let $A_1=\diag(1,0)$, $A_2=\diag(0,1)$ and $\bA=(A_1,A_2)$. Then
$$\cW_{\bA} =\d_{(1,0)}P_1 + \d_{(0,1)}P_2,\quad 
P_1:x\mapsto \left(\begin{matrix} x_1\\0\end{matrix}\right),\ 
P_2:x\mapsto \left(\begin{matrix} 0\\x_2\end{matrix}\right),\ x\in\R^2,$$
but $W(\bA) = K_{(1,1)}(\bA) =\co\left\{ \left(\begin{matrix} 1\\0\end{matrix}\right),
 \left(\begin{matrix} 0\\1\end{matrix}\right)\right\}$.
\item Let $A_1 = \left(\begin{matrix}0&1\\0&0\end{matrix}\right)$
and $A_2=\left(\begin{matrix}0&1\\0&1\end{matrix}\right)$. Then as in \cite[Example 5.31]{J}, we have
$$\cW_{\bA} 
=\left(\begin{matrix}0&
\delta_0'\otimes\chi_{[0,1]}\\
0&\delta_0\otimes\delta_1\end{matrix}\right) 
+ \left(\begin{matrix}0&
\delta_0\otimes(\delta_0-\delta_1)\\
0&\delta_0\otimes\delta_1\end{matrix}\right)$$
so $\supp(\cW_\bA) = \{0\}\times[0,1] = K(\bA)= W(\bA)$. The polynomial $P^\bA$ is not complete
\cite[Example 5.36]{J}. The spectral projections associated with $\av A\z$ for $\z\in\C^2$
have poles on the unit circle centred at zero---a phenomenon forbidden by Rellich's lemma in the
hermitian case \cite{JS}.

\item Not so obvious is that the operator valued distribution $\cW_\bA$ vanishes across the double tangent of the Kippenhahn curve $C(\bA)$
associated with
$$\bA = \left(\left(\begin{matrix}1&0&0\\0&-1&0\\0&0&-1\end{matrix}\right),\left(\begin{matrix}
0&0&1\\0&0&1\\1&1&0
\end{matrix}\right)\right).$$
The curve $C(\bA)$ is the third type in Kippenhahn's classification \cite[Theorem 26]{Kip}  of numerical range of 
$(3\times 3)$ matrices, see \cite[Figure X]{JS}. The argument goes as follows.
The couple $\bA=A_1+iA_2$ has no joint invariant subspaces otherwise $\supp(\cW_\bA)$ would be an
elliptical region or an elliptical region together with an outside  point---the joint eigenvalue, or just finitely many points corresponding to the case when $A_1+iA_2$ is a normal matrix, which it is not. 

The elliptical
region for two $(3\times 3)$ matrices $\vec B$ derives from a two dimensional joint invariant subspace $X$,
 should one exist, on which $\vec B$ is represented by
 $(2\times 2)$ hermitian matrices. Denote the restriction of $\vec B$ to the two dimensional joint invariant subspace $X$ by $\vec B_X$.
 The Pauli matrices $\bs\s$ and the identity $I_2$ form a linear basis of the  $(2\times 2)$ hermitian matrices, so 
 there exists an affine transformation $L:x\mapsto Tx+a$, $x\in\R^3$ with rank two  linear $T$  and $a\in\R^2$
 such that $\vec B_X = T\bs\s+aI_2$. By Theorem \ref{thm:monp} (i), $\cW_{\bB_X}$ is an affine
 transformation of $\cW_{\bs\s}$ and $\g(\vec B_X) = LS(\R^3)$ is an elliptical region in $\R^2$.
 However, $C(\bA)$ is a cardioid so the set $(N_\bA)^\circ\setminus\supp(\cW_\bA)$ is a nontrivial lacuna of $\cW_\bA$.

 To check the double tangent, observe that
 $$\z_0I +\langle\bs A,\bs\z\rangle = \begin{pmatrix}\z_0+\z_1&0&\z_2\\0&\z_0-\z_1&
 \z_2\\\z_2&\z_2&\z_0-\z_1\end{pmatrix}$$
 and
 \begin{eqnarray*}
P^\bA(\z) &=& (\z_0+\z_1)((\z_0-\z_1)^2-\z_2^2)-(\z_0-\z_1)\z_2^2\\
 &=& (\z_0^2-\z_1^2)(\z_0-\z_1)- 2\z_0\z_2^2.
\end{eqnarray*}
Then $P^\bA((1,1,0))=0$ and $P_{(1,1,0)}^\bA(\z) = 2((\z_0-\z_1)^2-\z_2^2)$, so a line parallel to 
$\{x_1=0\}$ is tangential to $C(\bA)$ at two points satisfying $x_2=\pm (x_1-1)$ in $\R^2$.
\end{enumerate}
\end{xmp}

We finish Part I of this paper with an explicit calculation of $E_\bA=\cF^{-1}\left(\frac{1}{P^\bA(\xi-i0)}\right)$
for two $(3\times 3)$ hermitian matrices $\bA$ with a joint eigenvalue. Here $E_\bA$ has a Petrovsky lacuna on which
$\cW_\bA$ vanishes. In this example, it is easy to write down the Weyl functional calculus:
$$\cW_{\bs A}=\delta_aP_a\oplus\cW_{(\s_1,\s_2)}P$$
for selfadjoint projections $P_a$, $P$ with $P_a+P=I$ on $\C^3$ and as mentioned above $\cW_{(\s_1,\s_2)}$
is the affine image of $\cW_{\bs\s}$ so $\supp(\cW_{\bA})$ is the union of $\{a\}$ and the closed unit disk centred at zero. The numerical
range $N_\bA$ is the convex hull of $\supp(\cW_{\bA})$. On the other hand calculating the fundamental solution $E_\bA$ takes more effort.

\begin{xmp} \label{Pauli2} Let $\bA= (a_1\oplus \s_1, a_2\oplus \s_2)$ for the Pauli matrices (\ref{eq:Paulim})
and $a\in\R^2$. Then
$$P^\bA(\xi) =\det(\xi_0I+\langle \bs A,\xi\rangle) = (\xi_0 + \langle a,\bs\xi\rangle)(\xi_0^2- |\bs\xi|^2),$$ 
The Kippenhahn curve $C(\bs A)$ is the union of the singleton $\{a\}$ and the unit circle centred at $0$
and $N_{\bs A}$ is the convex hull of $C(\bs A)$.

Then by \cite[equation (4.1)]{Ber} we have
\begin{align*}
\cF^{-1}\left(\frac{1}{P(\xi-i0)}\right) &= \frac{i}{2\pi^\frac12\G(\frac12)}H(x_0+\langle a,\vec x\rangle)\delta(\bx -x_0a)*H(x_0)(x_0^2-|\bs x|^2)_+^{-\frac12}\\
&= \frac{i}{2\pi^\frac12\G(\frac12)}\int_0^\infty H(x_0-y_0)(y_0^2-|\bs x-(x_0-y_0)a|^2)_+^{-\frac12}\,dy_0\\
&= \frac{i}{2\pi^\frac12G(\frac12)}\int_0^{x_0} (y_0^2-|\bs x-(x_0-y_0)a|^2)_+^{-\frac12}\,dy_0,
\end{align*}
where $H$ is the Heaviside function and $\delta$ is the unit point mass at zero, both interpreted as distributions.
The notation means that we are integrating the function 
$$y_0\longmapsto (y_0^2-|\bs x-(x_0-y_0)a|^2)_+^{-\frac12} $$
with values in the space of  distributions
over the interval $[0,x_0]$.

Suppose $x_0=1$,  $1 < |x| < |a|$. Then $\langle a,\bx-a\rangle < 0$ and $|\langle a,\bx-a\rangle| > \sqrt{|a|^2-1}|x-a| $
because the angle $\theta$ between $a$ and $x-a$ satisfies $\cos\theta >  \sqrt{|a|^2-1}/|a|$.
$$y_0^2-|\bs x-(1-y_0)a|^2 = -(|\bx-a|^2 +2\langle a,\bx-a\rangle y_0+(|a|^2-1)y_0^2) := p(y_0).$$
$$-(|a|^2-1)\big((y_0+\langle a,\bx-a\rangle/(|a|^2-1))^2 -(\langle a,\bx-a\rangle^2/(|a|^2-1)^2-|\bx-a|^2/(|a|^2-1))\big)$$
If $|\langle a,\bx-a\rangle| < \sqrt{|a|^2-1}|x-a| $, then the expression is negative and the integral is zero.
Suppose that $|\langle a,\bx-a\rangle| > \sqrt{|a|^2-1}|x-a| $.
$$y_0=\left(-\langle a,\bx-a\rangle\pm \sqrt{\langle a,\bx-a\rangle^2-(|a|^2-1)|\bx-a|^2}\right)\vee0/(|a|^2-1)$$
$$\langle a,\bx-a\rangle^2-(|a|^2-1)|\bx-a|^2 = (|a|^2\cos\theta - (|a|^2-1))|\bx-a|^2 > 0.$$

If $|x| > 1$, then the limits of integration are
$$(y_0+ \langle a,\bx-a\rangle/(|a|^2-1)) =\pm \sqrt{\langle a,\bx-a\rangle^2-(|a|^2-1)|\bx-a|^2}/(|a|^2-1).$$ 

If $|x| < 1$, then $-\langle a,\bx-a\rangle + \sqrt{\langle a,\bx-a\rangle^2-(|a|^2-1)|\bx-a|^2} > (|a|^2-1)$ and
$$0<-\langle a,\bx-a\rangle - \sqrt{\langle a,\bx-a\rangle^2-(|a|^2-1)|\bx-a|^2} < (|a|^2-1)$$
because the polynomial $p$ is equal to $1-|x|^2$ at $y_0=1$, so it has a zero to the right of 1.

If $|x| > 1$ and $|\langle a,\bx-a\rangle| > |a|^2-1$, that is, on the other side of the unit circle to $a$, then
the polynomial $p$ has zeros on the positive axis because it has the values $-|x-a|^2$ at $y_0=0$, 
$1-|x|^2$ at $y_0=1$ and the maximum value is at 
$$y_0 = |\langle a,\bx-a\rangle|/(|a|^2-1) > 1.$$
Hence the least zero of $p$ satisfies the inequality
$$1 < \frac{-\langle a,\bx-a\rangle- \sqrt{\langle a,\bx-a\rangle^2-(|a|^2-1)|\bx-a|^2}}{|a|^2-1} < 
\frac{|\langle a,\bx-a\rangle|}{|a|^2-1} $$
and $p(y_0)=y_0^2-|\bs x-(1-y_0)a|^2 < 0$ for all $0\le y_0 \le 1$. The polynomial $p$ is increasing on $[0,1]$
where $(y_0^2-|\bs x-(x_0-y_0)a|^2)_+$ vanishes.

Let $u(x) = \left(\sqrt{\langle a,\bx-a\rangle^2-(|a|^2-1)|\bx-a|^2}\right)/(|a|^2-1)$.
$$\int\frac1{\sqrt{b^2-x^2}}\,dx = \sin^{-1}(x/b)$$
Then for $|x|> 1$,  $\cos\theta >  \sqrt{|a|^2-1}/|a|$ and  $\langle a,\bx-a\rangle < 0$, the integral 
$$\int_0^{1} (y_0^2-|\bs x-(x_0-y_0)a|^2)_+^{-\frac12}\,dy_0$$
 equals
$$\left [\sin^{-1}\left(t(|a|^2-1)/\sqrt{\langle a,\bx-a\rangle^2-(|a|^2-1)|\bx-a|^2}\right)\right]_{t=-u(x)}^{t=u(x)}=\pi.$$
For $|x| \le 1$, $\int_0^{1} (y_0^2-|\bs x-(x_0-y_0)a|^2)_+^{-\frac12}\,dy_0$ equals
$$\left [\sin^{-1}\left(t(|a|^2-1)/\sqrt{\langle a,\bx-a\rangle^2-(|a|^2-1)|\bx-a|^2}\right)\right]_{t=-u(x)}^{t=1+\langle a,\bx-a\rangle/(|a|^2-1)}$$
$$=\pi/2+\sin^{-1}\left((\langle a,\bx\rangle-1)/\sqrt{\langle a,\bx-a\rangle^2-(|a|^2-1)|\bx-a|^2}\right).$$
Also $\cW_{\bs A}=\delta_aP_a\oplus\cW_{(\s_1,\s_2)}P_2$

%

To calculate the wave front set, we note that
$$(\xi_0+t\z_0)^2- |\bs\xi+t\bs\z|^2 =\xi_0^2- |\bs\xi|^2 +2t(\xi_0\z_0-\langle \bs\xi,\bs\z\rangle)+t^2(\z_0^2-|\bs\z|^2).$$
For $\xi_0^2- |\bs\xi|^2=0$, $\eta = \pm|\bs\xi|e_0-\bs\xi$,  we have $\G_\eta = \{\pm\langle \eta,\z \rangle > 0\}$,
$K_\eta = \{\pm t\eta | t > 0\}$. $\cup_{\eta \ne 0} K_\eta$ is a circular cone in $\R^3$ and
\begin{align*}
P(\xi) &=(\xi_0 + \langle a,\bs\xi\rangle)(\xi_0^2- |\bs\xi|^2)\\
\G(P) &=  \{\xi_0 + \langle a,\bs\xi\rangle >0,\ \xi_0- |\bs\xi|>0,\ \xi_0+ |\bs\xi|>0\}\\
K(P) &=  \text{co}\{t(e_0+a),t(e_0+\xi):|\xi|=1,\xi\in\R^n,\ t>0\ \}
\end{align*}
$$P(\xi+t\z) = (\xi_0^2- |\bs\xi|^2)(\xi_0 + \langle a,\bs\xi\rangle) +2t(\frac12(\xi_0^2- |\bs\xi|^2)(\z_0 + \langle a,\bs\z\rangle)+(\xi_0\z_0-\langle \bs\xi,\bs\z\rangle)(\xi_0 + \langle a,\bs\xi\rangle))$$
$$+t^2((\z_0^2-|\bs\z|^2)(\xi_0 + \langle a,\bs\xi\rangle)+2(\xi_0\z_0-\langle \bs\xi,\bs\z\rangle)(\z_0 + \langle a,\bs\z\rangle)) +t^3(\z_0^2-|\bs\z|^2)(\z_0 + \langle a,\bs\z\rangle)$$
$$\xi_0 + \langle a,\bs\xi\rangle=0,\ \eta = e_0+a.$$
$$|a| > 1,\ \xi_0^2 = |\bs\xi|^2,\ \langle a,\bs\xi\rangle =  -\xi_0.$$
$$P_\xi(\z)=2(\xi_0\z_0-\langle \bs\xi,\bs\z\rangle)(\z_0 + \langle a,\bs\z\rangle)$$
$$\Xi_{\xi'}(\z) = \{\z:|\bs\xi'|\z_0-\langle \bs\xi',\bs\z\rangle=0\}\cup\{\z_0 + \langle a,\bs\z\rangle=0\},\ \xi_0=|\bs\xi'|$$
Because $-\bs\xi' = \bs\xi$ for one of the two solutions of $\langle a,\bs\xi\rangle =  |\bs\xi|$, this set is equal to
$$\Xi_\xi(\z) = \{\z:|\bs\xi|\z_0+\langle \bs\xi,\bs\z\rangle=0\}\cup\{\z_0 + \langle a,\bs\z\rangle=0\},\ \xi_0=-|\bs\xi|\}$$
Now choose $\ |\bs \xi|=1$. Then 
\begin{align*}
\G_\xi &=\{\z : \z_0 + \langle \bs\xi,\bs\z\rangle > 0,\ \z_0 + \langle a,\bs\z\rangle>0 \},\\
K_\xi &=\text{co}\{\bs \xi ,a\},\ \xi_0=-1,
\end{align*} 
for each of the two solutions of $\langle a,\bs\xi\rangle =  1.$
The wave front set
$W(\bs A)$ is the unit circle plus $\{a\}$ as well as segments tangential to the circle meeting at $\{a\}$.
\end{xmp}


\bigskip


\bibliographystyle{plain}

\end{document}